\documentclass[11pt,draft]{article}

\usepackage{mathtools}
\usepackage{amssymb}
\usepackage{amsthm}
\usepackage{xcolor}
\usepackage{tikz}
\usepackage{url}
\usepackage{fullpage}

\DeclareSymbolFont{bbold}{U}{bbold}{m}{n}
\DeclareSymbolFontAlphabet{\mathbbold}{bbold}

\newcommand{\R}{\mathbb{R}}

\newcommand{\1}{\mathbbold{1}}

\renewcommand{\epsilon}{\varepsilon}



\DeclareMathOperator{\dist}{dist}

\DeclareMathOperator{\av}{\tt av}

\renewcommand{\Re}{\operatorname{Re}}

\DeclareMathAccent{\Circ}{\mathalpha}{operators}{"17}

\let\eps\varepsilon
\let\phi\varphi
\let\le\leqslant
\let\leq\leqslant
\let\ge\geqslant
\let\geq\geqslant

\makeatletter
\def\@row#1,{#1\@ifnextchar;{\@gobble}{&\@row}}
\def\@matrix{%
    \expandafter\@row\my@arg,;%
    \@ifnextchar({\\ \get@in@paren{\@matrix}}{\after@matrix}%
    }
\def\matrixtype#1#2#3{%
    \ifmmode\def\after@matrix{\end{#2}\right#3}%
    \else\def\after@matrix{\end{#2}\right#3$}$\f{\rm i}iffalse$\fi
    \left#1\begin{#2}\get@in@paren{\@matrix}%
    }
\def\@column#1,{#1\@ifnextchar;{\@gobble}{\\ \@column}}
\newcommand\vect{}
\def\svect(#1){\left(\begin{smallmatrix}\@column#1,;\end{smallmatrix}\right)}
\def\vect{\get@in@paren{\@vect}}
\def\@vect{\left(\begin{matrix}\expandafter\@column\my@arg,;\end{matrix}\right)}
\def\get@in@paren#1({\def\my@arg{}\def\my@rest{}\def\after@get{#1}\get@arg}
\let\e@a\expandafter
\def\get@arg#1){\e@a\kl@test\my@rest#1(;}
\def\kl@test#1(#2;{\e@a\def\e@a\my@arg\e@a{\my@arg#1}%
                   \ifx:#2:\let\my@exec\after@get
                   \else\let\my@exec\get@arg
                        \e@a\def\e@a\my@arg\e@a{\my@arg(}%
                        \def@rest#2;%
                   \f{\rm i}my@exec}
\def\def@rest#1(;{\def\my@rest{#1\kl@zu}}
\def\kl@zu{)}

\makeatletter
\newcommand\MyPairedDelimiter{%
  \@ifstar{\My@Paired@Delimiter{{}}}
          {\My@Paired@Delimiter{}}%
}
\newcommand\My@Paired@Delimiter[4]{%
  \newcommand#2{%
    \@ifstar{\start@PD{#1}{\delimitershortfall=-1sp}{#3}{#4}}
            {\start@PD{#1}{}{#3}{#4}}%
  }%
}
\newcommand\start@PD[5]{%
  #1\mathopen{\mathpalette\put@delim@helper{\put@delim{#2}{#3}{.}{#5}}}%
  #5%
  \mathclose{\mathpalette\put@delim@helper{\put@delim{#2}{.}{#4}{#5}}}%
}
\newcommand\put@delim@helper[2]{%
  \hbox{$\m@th\nulldelimiterspace=0pt #2#1$}%
}
\newcommand\put@delim[5]{%
  \setbox\z@\hbox{$\m@th#5{#4}$}%
  \setbox\tw@\null
  \ht\tw@\ht\z@ \dp\tw@\dp\z@
  #1#5%
  \left#2\box\tw@\right#3%
}

\makeatother
\MyPairedDelimiter*{\abs}{\lvert}{\rvert}
\MyPairedDelimiter*{\norm}{\lVert}{\rVert}
\MyPairedDelimiter{\set}{\{}{\}}

\theoremstyle{plain} 
\newtheorem{theorem}{Theorem}[section]
\newtheorem{corollary}[theorem]{Corollary}
\newtheorem{lemma}[theorem]{Lemma}
\newtheorem{proposition}[theorem]{Proposition}
\theoremstyle{definition}

\newtheorem*{definition}{Definition}
\newtheorem{remark}[theorem]{Remark}

\usepackage{enumitem}

\setenumerate[1]{nolistsep} 
\setenumerate[2]{nolistsep} 

\begin{document}

\medmuskip=4mu plus 2mu minus 3mu
\thickmuskip=5mu plus 3mu minus 1mu
\belowdisplayshortskip=9pt plus 3pt minus 5pt

\title{\bf Resolvent estimates in homogenisation of periodic problems of fractional elasticity}

\author{Kirill Cherednichenko and Marcus Waurick}

\date{}

\maketitle

\begin{abstract}
 We provide operator-norm convergence estimates for solutions to a time-dependent equation of fractional elasticity in one spatial dimension, with rapidly oscillating coefficients that represent the material properties of a viscoelastic composite medium. Assuming periodicity in the coefficients, we prove operator-norm convergence estimates for an operator fibre decomposition obtained by applying to the original fractional elasticity problem the Fourier--Laplace transform in time and Gelfand transform in space. 
  We obtain estimates on each fibre that are uniform in the quasimomentum of the decomposition and in the period of oscillations of the coefficients as well as quadratic with respect to the spectral variable. On the basis of these uniform estimates we derive operator-norm-type convergence estimates for the original fractional elasticity problem, for a class of sufficiently smooth densities of applied forces. 
\end{abstract}



\section{Introduction}

One notable direction in the recent mathematical literature on the derivation of the overall behaviour of composites, in the context of linearised elasticity, elastodynamics, electrodynamics, is the asymptotic analysis, as the ratio $\varepsilon$ of the microstructure size to the macroscopic size of the material sample goes to zero, of the resolvents (or ``solution operators'') of the (conservative) operators, elliptic in space or hyperbolic in space-time, that describe the response of the composite to exterior forces. Whenever convergence with respect to the operator norm is proved, one can often infer a host of properties about the underlying time dependent problem, in particular, the behaviour of the spectrum ({\it i.e.} the response to time-harmonic waves of certain frequencies and wave packets) and the convergence of the corresponding spectral projectors and operator semigroups (a version of the Trotter-Kato theorem). In the periodic setting, the advance in this kind of questions has been possible thanks to the Floquet-Bloch decomposition of the original operators into direct fibre integrals with respect to the ``quasimomentum" $\theta$ and the development of various tools combining operator-theoretic considerations and the error estimates, in the spirit of classical asymptotic analysis, that are uniform in $\theta.$ The revised notion of homogenisation as the asymptotic procedure of replacing the original resolvent family by operators where different spatial scales are separated, in the sense of being described by a system of coupled field equations, can be viewed as a rigorous generalised procedure of classifying composites according to their overall response. For example, periodic composites whose component materials have highly contrasting properties ({\it e.g.} in the form of ``soft" inclusions embedded in a ``stiff'' matrix) can  be seen, using this kind of approach, to exhibit physical properties recognisable as those of so-called ``metamaterials'', {\it e.g.} media with negative refractive index, artificial magnetism {\it etc}. The asymptotically equivalent operator family found in this procedure can be viewed as an alternative model for the same composite, equivalent to it  in the sense of respecting all of its qualitative and quantitative properties. 

In the present work we address, for the first time, the operator-norm homogenisation-type estimates for a family of non-conservative time-dependent problems, where the energy dissipation takes place internally by friction-like forces into heat. We consider the linearised problem of one-dimensional elasticity, modified by an operator of fractional time differentiation, see \cite[Section 4.2]{Waurick2014MMAS_Frac}. For the unknown scalar valued-functions $u$ and $\sigma,$ which represent the elastic displacement and stress at the point $x$ of the medium at time $t,$ we consider the problem
\[
   \begin{cases}
      \partial_t^2 u - \partial_x\sigma = f,&\\[0.2em]
      \sigma = (C + \partial_t^\alpha D)\partial_x u.
   \end{cases} 
\]
Here $C,$ $D$ are non-negative functions depending on the spatial variable $x\in{\mathbb R}$ only, which can be viewed as viscoelastic constitutive parameters of the medium, $\alpha\in (0,1]$, and $f$ is a given source term describing the density of forces applied to the medium. The operator $\partial_t^\alpha$ is the fractional time derivative in a sense to be described in the next section. If the support of $f$ with respect to the temporal variable is bounded below, then $\partial_t^\alpha$ coincides with the Riemann--Liouville derivative (see \emph{e.g.} \cite{MillerRoss93,Podlubny99}). We refer to \cite{Bertram} for a justification of the model to describe viscoelastic behaviour from an engineering perspective. 

The well-posedness of the above dynamic problem has been addressed in \cite{Waurick2014MMAS_Frac}, and in \cite{Waurick2014SIAM_HomFrac} a corresponding homogenisation problem has been considered, where convergence of the corresponding solution operators is established in a certain weak topology for operators in $L^2$-spaces with appropriate weights that ensure solvability and well-posedness of the original heterogeneous problems.  More precisely, assuming periodicity and boundedness in the coefficients $C$ and $D$, it has been shown in \cite[Theorem 3.6]{Waurick2014SIAM_HomFrac} that for  $\alpha\geq 1/2$ 
the solution operators for the problems
\begin{equation}
   \begin{cases}
      \partial_t^2 u_n - \partial_x \sigma_n(u_n) = f,&\\[0.2em]
      \sigma_n(u_n) = \bigl(C(n\cdot) + \partial_t^\alpha D(n\cdot)\bigr)\partial_{x,0} u_n,\ \ \ \ n\in{\mathbb N},
   \end{cases}
 \label{n_dependent}
\end{equation}
on the time-space domain $\mathbb{R}\times (0,1)$
converge as $n\to\infty$ in the weak operator topology of a
Hilbert space of functions defined in space-time to the solution operator for 
\[
 \begin{cases}{\displaystyle
    \partial_t^2 u -\partial_x\sigma^{\rm hom}(u)= f},\\[0.4em] {\displaystyle
   \sigma^{\rm hom}(u):=\left(\int_0^1D^{-1}\right)^{-1}\partial_t^\alpha\partial_{x,0}u} 
   {\displaystyle 
   +\partial_t^\alpha\sum\limits_{k=1}^\infty\left(-\sum\limits_{\ell=1}^\infty (-\partial_t^{-\alpha})^\ell \int_0^1C^\ell D^{-\ell-1}\left(\int_0^1D^{-1}\right)^{-1} \right)^{k}\partial_{x,0}u.
   }
   \end{cases}
\]
Here $\partial_{x,0}$ denotes the distributional derivative on the interval $(0,1)$ with zero-boundary conditions at the endpoints $x=0$ and $x=1.$
In \cite[Theorem 3.6]{Waurick2014SIAM_HomFrac} the boundedness of the underlying spatial domain 
has been crucial for the analysis of the problem (\ref{n_dependent}), in order to ensure a compactness condition assumed in a homogenisation theorem of more general nature, see  \cite[Theorem 4.1]{Waurick2014SIAM_HomFrac}, \cite[Theorem 3.5]{Waurick2013AA_Hom}, or \cite[Theorem 5.2.3]{Waurick2016Hab}. 

In the present article we complement the results of \cite{Waurick2014SIAM_HomFrac} by passing from the case of a bounded spatial domain $(0,1)$ to the whole space $\mathbb{R}$. Moreover, we shall provide resolvent convergence estimates for problems of the form (\ref{n_dependent}), rather than qualitative convergence results. The convergence analysis for resolvents for homogenisation problems goes back to the works \cite{Sevostianova}, \cite{Zhikov1989}, where the behaviour of the Green functions for parabolic equations with rapid oscillations was studied using their decomposition with respect to systems of eigenfunction for certain basic problems on the microscale, known as ``cell problems''. More recently, an operator-theoretic version of this approach was developed in \cite{Birman_Suslina2004}, based on a combination of the methods of spectral and perturbation theory applied to a class of spatial self-adjoint operators of the form $X_\theta^*X_\theta,$ where $X_\theta$ is a linear operator pencil. The approach of \cite{Birman_Suslina2004} was combined with boundary-layer analysis in \cite{Suslina_Neumann}, \cite{Suslina_Dirichlet}, where resolvent estimates for elliptic problems in bounded domains have been obtained. It was further refined in \cite{Suslina_Meshkova}, where the dependence of the error estimates on the spectral parameter was investigated and operator-norm convergence estimates for semigroups (equivalently, parabolic time-dependent problems) have been proved by expressing the semigroup in terms of the contour integrals of the resolvents. At the same time, several other approaches have been developed to obtain order-sharp operator-norm estimates in the elliptic self-adjoint and parabolic contexts: the method of first-order approximation in \cite{Zhikov_Pastukhova1}, \cite{Zhikov_Pastukhova2}, the method of periodic unfolding 
in \cite{Griso} and an adaptation of the classical boundary-layer potential analysis for systems of PDEs in 
\cite{Kenig_et_al1}, \cite{Kenig_et_al2}. The main advantage of the operator-norm asymptotic estimates obtained in these works in comparison to some of the earlier approaches, {\it e.g.} two-scale convergence, is that they automatically imply that the energy convergence criterion is satisfied: for problem data from a wide class the total energy of the solution to the limit problem is the limit of the total energies described by the original parameter-dependent problems. The task of ensuring the convergence of the total energy becomes even more challenging in problems where the underlying spatial operator is elliptic but not in a uniform way with respect to the microstructure size, as it happens for example in problems with coefficients degenerating on a part of the unit cell, see \cite{Zhikov2000}, \cite{Smyshlyaev_Mechanics}, as the positive answer to it is then more sensitive to the operator topology chosen for the convergence statement. This can already be seen through some simple examples that require the use of multiscale, rather than classical ``one-scale", techniques.

We want to emphasise that the approach developed here contains the parabolic heat equation as well as the hyperbolic wave equation as respective special cases. In fact, for $\alpha=1$ and $C=0$, one recovers the heat equation, and for $D=0$ the wave equation is treated.  Note that for $C=0$ and $\alpha\in(0,1)$ we consider the case of so-called superdiffusion equations as well, see \cite{Waurick2013AA_Hom}. Moreover, see \cite{Waurick2016AML}, we also treat mixed type equations, that is, equations changing its type on the underlying spatial domain.

We next outline the structure of the paper. In Section \ref{well_pos_section}, we recall the formulation for evolutionary problems using weighted spaces with respect to time and the corresponding solution theory. This includes the study of the Fourier--Laplace transform of the evolutionary problem of viscoelasticity. In Section \ref{Gelfand_section} we outline the class of periodic problems we aim to address in the viscoelasticity context, describe a more general class of problems that can be treated using the same approach and formulate the main result of the present article (Theorem \ref{t:mt}) for this class. 

All remaining sections, apart from the last one, are devoted to a proof of our main theorem. In particular, in Section \ref{sec:ortho_decom} we carry out a spectral decomposition of the one-dimensional derivative on a bounded interval with periodic boundary conditions. This decomposition is needed for a reformulation of the unbounded spatial operator of the viscoelasticity problem following the application of the Gelfand transform of Section \ref{Gelfand_section}. Further, in Section \ref{averaging} we discuss the relationship of the spectral decomposition derived in Section \ref{sec:ortho_decom} to the averaging operator defined in Section \ref{Gelfand_section}, which serves as a means to compute the integral average as an action on suitable operator spaces. Finally, Section \ref{proof_thm} contains the proof of our main result. The article is concluded with Section \ref{viscoelasticity_treatment}, where we come back to the viscoelasticity problem that motivated our study. We apply the general norm-resolvent estimates obtained in Sections \ref{Gelfand_section}--\ref{proof_thm} to derive uniform operator estimates for the homogenisation problem of the viscoelasticity system stated in \eqref{n_dependent}, with the underlying spatial domain being the whole line $\mathbb{R}$ rather than a bounded interval.

\section{Well-posedness of the dynamic problem}
\label{well_pos_section}

In this section we will briefly recall the well-posedness result for the dynamic problem of fractional elasticity, which was obtained in \cite{Waurick2014MMAS_Frac}. We first outline the related functional analytic framework for the operator of fractional time derivative. A more detailed exposition of this setting can be found in \cite[Section 2]{Waurick2013JDDE_Delay} or \cite[Section 2.1]{Waurick2014MMAS_Frac}. 

For $\nu>0$ and a Hilbert space $H$, we define 
\[
   L_\nu^2(\mathbb{R};H)\coloneqq\left\{ f\in L_{\text{loc}}^2(\mathbb{R};H):\ \|f\|_{L_\nu^2({\mathbb R}; H)}^2\coloneqq \int_{\R}\bigl\|f(t)\bigr\|_H^2\exp(-2t\nu)dt<\infty\right\}.
\]
We denote by $H_\nu^1(\mathbb{R};H)$ the space of weakly differentiable $L_\nu^2$-functions with derivative in $L_\nu^2$. It is shown that the operator
\[
   \partial_t \colon  L_\nu^2(\mathbb{R};H)\supseteq H_\nu^1(\mathbb{R};H)\ni f \mapsto f' \in L_\nu^2(\mathbb{R};H),
\]
is continuously invertible, with the inverse given by the Bochner integral
\[
  \partial_t^{-1}f (t) = \int_{-\infty}^t f,\quad \ \ t\in \R,\ f\in L_\nu^2(\mathbb{R};H),
\]
and that $\|\partial_{t}^{-1}\|\leq 1/\nu$. Denote by $m$ the operator in $L^2(\mathbb{R};H)$ of multiplication by the independent variable:
\[
   m \colon L^2(\mathbb{R};H)\supseteq\bigl\{ f\in L^2(\mathbb{R};H): \xi\mapsto \xi f(\xi) \in L^2(\mathbb{R};H)\bigr\}\ni f\mapsto  \bigl(\xi\mapsto \xi f(\xi)\bigr)\in L^2(\mathbb{R};H).
\]
Further, we introduce the Fourier--Laplace transform $\mathcal{L}_\nu \colon L_\nu^2(\mathbb{R};H)\to L^2(\mathbb{R};H)$ as the unitary extension of the mapping 
\begin{equation}
  \mathcal{L}_\nu \phi(\xi) \coloneqq \frac{1}{\sqrt{2\pi}}\int_\mathbb{R} e^{-{\rm i}\xi t-\nu t}\phi(t)dt, \quad \xi\in \mathbb{R},\ \phi\in C_{\rm c}^\infty(\R;H),
 \label{Laplace_transform}
\end{equation}
where $C_{\rm c}^\infty(\R;H)$ is the set of $H$-valued smooth functions with compact support. It is a consequence of the spectral theorem for the distributional derivative in $L^2(\mathbb{R}; H)$, that 
\[
\partial_t = \mathcal{L}_\nu^*({\rm i}m +\nu)\mathcal{L}_\nu,
\]
which allows us to define the ``fractional" time derivative of order $\alpha\in(0,1]:$  
\[
  \partial_t^\alpha \coloneqq \mathcal{L}_\nu^*({\rm i}m +\nu)^\alpha\mathcal{L}_\nu,
\]
or, more generally, the function $\mathfrak{M}$ of the operator $\partial_t:$
\[
\mathfrak{M}(\partial_t) \coloneqq \mathcal{L}_\nu^*\mathfrak{M}({\rm i}m +\nu)\mathcal{L}_\nu,
\]
for every analytic ${\mathbb L}(H)$-valued function $\mathfrak{M}$ defined on ${\rm i}\mathbb{R}+\mathbb{R}_{\geq \nu}$.\footnote{Note that for the definition of $\mathfrak{M}(\partial_t)$ alone it is not necessary to have $\mathfrak{M}$ defined and analytic on a full right-half plane. 
However, in order to maintain causality of a solution operator to certain abstract PDEs involving $\mathfrak{M}(\partial_t)$ analyticity on a right-half plane is an essential requirement, see also \cite[Remark 2.11]{PicPhy}.} Here we denote by ${\mathbb L}(H)$ the set of bounded  linear operators on $H,$
as well as 
\begin{equation}
   (\mathfrak{M}({\rm i}m +\nu)\phi)(\xi)\coloneqq \mathfrak{M}({\rm i}\xi +\nu)\phi(\xi),\quad \phi\in C_{\rm c}(\mathbb{R};H),\quad \xi\in \mathbb{R}.
 \label{M_mathfrak}
\end{equation}
In the special case $\mathfrak{M}\colon z \mapsto z^\alpha,$ see \cite[p. 3143]{Waurick2014MMAS_Frac}, 
the formula (\ref{M_mathfrak}) yields a proper implementation of the Riemann--Liouville derivative, see \emph{e.g.} \cite{MillerRoss93,Podlubny99}.
Now we are in a position to recall the solution theory for the dynamic problem, which will be eventually applied to the model of viscoelasticity introduced above. It is worth mentioning possible generalisations of this approach to the non-linear and/or non-autonomous case \cite{Trostorff2012,Trostorff2014a,RainerPicard2013,Waurick2014MMAS_Non,Waurick2016Hab}. In the following, we will use the same notation for an operator in $H$ and for the corresponding abstract multiplication operator in $L_\nu^2(\mathbb{R};H)$.

\begin{theorem}[{{\cite[Solution Theory]{PicPhy}}}]\label{t:st} Suppose that $\nu>\mu>0$, and let ${\mathfrak M}\colon {\rm i}\mathbb{R}+\mathbb{R}_{>\mu}\to {\mathbb L}(H)$ be bounded and analytic. Assume that there is $\gamma>0$ such that for all $z\in  {\rm i}\mathbb{R}+\mathbb{R}_{\geq \nu}$ we have
\[
\Re\bigl(z{\mathfrak M}(z)\bigr)-\gamma I\ge 0,
\]
where $\Re\bigl(z{\mathfrak M}(z)\bigr):=\bigl(z{\mathfrak M}(z)+z^*{\mathfrak M}(z)^*\bigr)/2$ and $I$ is the identity operator on $H.$ Consider also a skew-selfadjoint operator $A$ in $H.$

Then the operator 
\[
   B\coloneqq \partial_t{\mathfrak M}(\partial_t)+A \colon  L_\nu^2(\mathbb{R};H)\supseteq D(\partial_t)\cap D(A)\to L_\nu^2(\mathbb{R};H)
\]
is densely defined and closable. Furthermore, for its closure $\overline{B}$ one has $0\in \rho(\overline{B})$, $\bigl\|\overline{B}^{-1}\bigr\|\leq1/\gamma$, and $S\coloneqq \overline{B}^{-1}$ is causal ({\it cf.} \cite{Waurick2016Hab}), that is, for all $t\in \mathbb{R}$, the property
\[
   \1_{(-\infty,t]}S\1_{(-\infty,t]}=\1_{(-\infty,t]}S
\]
holds, where $\1_{(-\infty,t]}$ is the operator of multiplication by the characteristic function of the time interval $(-\infty,t].$
\end{theorem}

As a corollary of Theorem \ref{t:st}, we obtain the following statement about well-posedness of the dynamic problem of fractional elasticity.

\begin{corollary}[{{\cite[Theorem 4.1]{Waurick2014MMAS_Frac}}}]\label{c:st} Let $\nu_0>0$, $C,D\in L^\infty(\mathbb{R};\mathbb{R}_{\geq 0})$, $\alpha\in (0,1]$. Assume that for some $c>0$ one has
\[
   C(x)+D(x)\nu_0^\alpha\geq c\quad{\rm a.e.}\ x\in \mathbb{R},
\]
and let $A$ be a skew-selfadjoint operator in $\bigl[L^2(\Omega)\bigr]^2$ for some open $\Omega\subseteq \mathbb{R}$.
Then there exists $\nu\geq \nu_0$ such that the operator
\begin{equation}
   \partial_t\begin{pmatrix}
      1 & 0 \\[0.3em] 0 & \bigl(C+D\partial_t^\alpha\bigr)^{-1}
   \end{pmatrix} + A
   \label{corollary_op}
\end{equation}
is densely defined and closable, with continuously invertible closure in $L_\nu^2\bigl(\mathbb{R};\bigl[L^2(\Omega)\bigr]^2\bigr)$.
\end{corollary}
 The essential part in the proof of the above corollary, which is presented in \cite{Waurick2014MMAS_Frac}, is to verify that the operator-valued function 
 \[
    {\mathfrak M}(z)\coloneqq \begin{pmatrix}
      1 & 0 \\[0.3em] 0 & (C+Dz^\alpha)^{-1}
   \end{pmatrix},\ \ \ \ \ \ z\in{\rm i}{\mathbb R}+{\mathbb R}_{>\mu},
 \]
 satisfies the conditions
 of Theorem \ref{t:st} for all $\nu>\mu$ for some $\mu\geq\nu_0$.

\begin{remark} The original problem of fractional elasticity is obtained by setting $A=- \begin{pmatrix} 0 & \partial_x \\ \partial_x & 0\end{pmatrix}$ with $\Omega=\mathbb{R}$. Other boundary conditions are also possible. For instance, one can take $\Omega=(0,1)$ and $A=- \begin{pmatrix} 0 & \partial_x \\  \partial_{x,0} & 0\end{pmatrix},$ where $\partial_{x,0}$ denotes the distributional derivative in $L^2(0,1)$ defined on $H_0^1(0,1)$, so that $\partial_x=-\partial_{x,0}^*$. Further, the setting $\Omega= (0,1)$ and $A=- \begin{pmatrix} 0 & \partial_{x,\#} \\ \partial_{x,\#} & 0\end{pmatrix},$ with $\partial_{x,\#}\subseteq \partial_x$ and $H_\#^1(0,1)\coloneqq D(\partial_{x,\#})=\bigl\{ f\in H^1(0,1): f(0)=f(1)\bigr\},$ leads to another skew-selfadjoint realisation of the operator $A$. 
\end{remark}

\begin{remark}\label{r:rp} Here we comment on the overall strategy of the proof of Theorem \ref{t:st}.
Employing the Fourier--Laplace transform defined by the formula (\ref{Laplace_transform}), we obtain an operator in $L^2\bigl({\mathbb R}; \bigl[L^2(\Omega)\bigr]^2 \bigr)$ that is unitarily equivalent to (\ref{corollary_op}):
\begin{equation}
     \partial_t\begin{pmatrix}
      1 & 0 \\[0.3em] 0 & \bigl(C+D\partial_t^\alpha\bigr)^{-1}
   \end{pmatrix} + A
   = \mathcal{L}_\nu^* \Biggl( ({\rm i}m+\nu)\begin{pmatrix}
      1 & 0 \\[0.3em] 0 & \bigl(C+D({\rm i}m+\nu)^\alpha\bigr)^{-1}
   \end{pmatrix} + A
   \Biggr)\mathcal{L}_\nu.
   \end{equation}
The key idea for justifying the continuous invertibility of the (closure of) the latter operator is therefore to guarantee the continuous invertibility of the expression 
\begin{equation}
({\rm i}\xi+\nu)\begin{pmatrix}
      1 & 0 \\[0.3em] 0 & \bigl(C+D({\rm i}\xi+\nu)^\alpha\bigr)^{-1}
   \end{pmatrix} + A
   \label{first_order_op}
\end{equation}
as an operator in $\bigl[L^2(\Omega)\bigr]^2$ with a bound on the operator norm of its inverse that is uniform in $\nu\ge \mu$ and $\xi\in{\mathbb R}.$ To invert (\ref{first_order_op}) amounts to solving a certain resolvent problem. In the following, this problem will be the starting point for our operator-norm analysis, in the case when the coefficients $C,$ $D$ are periodic.
We summarise the above in the following proposition.

\begin{proposition}\label{t:sts} Suppose that $\nu_0>0$, $C,D\in L^\infty(\mathbb{R};\mathbb{R}_{\geq 0})$, $\alpha\in (0,1]$. Assume that for some $c>0$ one has
\[
   C(x)+D(x)\nu_0^\alpha\geq c\quad \ \ a.e.\ \ x\in \mathbb{R},
\]
and suppose that $A$ is a skew-selfadjoint operator in $\bigl[L^2(\Omega)\bigr]^2$, where $\Omega\subseteq \mathbb{R}$ is open.
Then there exists $\mu\geq\nu_0$ such that for all $\nu>\mu$ and $\xi\in \mathbb{R}$ the operator
\begin{equation}
   ({\rm i}\xi+\nu)\begin{pmatrix}
      1 & 0 \\[0.3em] 0 & \bigl(C+D({\rm i}\xi+\nu)^\alpha\bigr)^{-1}
   \end{pmatrix} + A
   \label{operator_class}
\end{equation}
is densely defined and continuously invertible in $\bigl[L^2(\Omega)\bigr]^2$. Moreover, there exists $\kappa>0$ such that for all $\xi\in{\mathbb R}$ the estimate
\[
    \sup_{
    \nu>\mu}\Bigg\|
    \Bigg(
   ({\rm i}\xi+\nu)\begin{pmatrix}
      1 & 0 \\[0.3em] 0 & \bigl(C+D({\rm i}\xi+\nu)^\alpha\bigr)^{-1}
   \end{pmatrix} + A
   \Bigg)
   ^{-1}\Bigg\|_{{\mathbb L}([L^2(\Omega)]^2)}\leq \kappa
\]
holds.
\end{proposition}

\end{remark}

\section{The resolvent problem, the Gelfand transform, and the main result}
\label{Gelfand_section}

In what follows we consider parameter-dependent families of operators described by expressions of the form (\ref{operator_class}). Aiming for quantitative results in periodic homogenisation theory, we restrict ourselves to the case when the coefficients $C, D$  are periodic and the parameter in the problem, which we denote by $\varepsilon,$ represents  their period ({\it cf.} parameter $n$ in (\ref{n_dependent})). We denote by $\widehat{C},$ 
$\widehat{D}$ the 1-periodic functions related to $C=C^\varepsilon,$ $D=D^\varepsilon$ by the formulae
\begin{equation}
\widehat{C}(\cdot)=C^\varepsilon(\varepsilon\cdot),\ \ \ \ \ \widehat{D}(\cdot)=D^\varepsilon(\varepsilon\cdot).
\label{CD_scaling}
\end{equation}

\begin{definition}\label{d:cf}
We consider the space
\[
   L_\#^\infty(\mathbb{R})\coloneqq\bigl\{ f\in L^\infty(\mathbb{R}): f(\cdot + 1)=f(\cdot)\bigr\}
\]
and for $\gamma>0$ define the set
\[
   \mathcal{M}_\gamma\coloneqq\bigl\{ M\in\bigl[L_\#^\infty(\mathbb{R})\bigr]^{2\times 2}:\ \Re M(x)\geq \gamma 1_{2\times 2}\ {\rm a.e.}\ x\in \mathbb{R}\bigr\}
\]
as well as a mapping ${\tt av}: {\mathcal M}_\gamma\to{\mathbb C}^{2\times 2},$ by the formula
\[
\av(M)\coloneqq \int_0^1 M,\quad\quad M\in \mathcal{M}_\gamma.
\]
\end{definition}

Suppose that $\varepsilon$-periodic functions $C^\varepsilon,$ $D^\varepsilon$ satisfy the conditions of 
Corollary \ref{c:st} for some $\nu_0>0,$ and consider the 1-periodic functions $\widehat{C}, \widehat{D}\in L_\#^\infty({\mathbb R})$ that are related to $C^\varepsilon,$ $D^\varepsilon$ for each $\varepsilon>0$ by the formulae (\ref{CD_scaling}). Notice that there exists $\mu\ge\nu_0$ such that for all $\nu>\mu$ and $\xi\in{\mathbb R}$ one has ({\it cf.} the discussion after Corollary \ref{c:st})
\begin{equation}
({\rm i}\xi+\nu)\begin{pmatrix}
      1 & 0 \\[0.3em] 0 & \bigl(\widehat{C}+\widehat{D}({\rm i}\xi+\nu)^\alpha\bigr)^{-1}
   \end{pmatrix} \in \mathcal{M}_\gamma,
\label{viscoelastic_matrix}
\end{equation}
for some $\gamma=\gamma(\nu_0).$ In Section  \ref{viscoelasticity_treatment}, where we apply the estimates of Theorem \ref{t:mt} below to the periodic viscoelasticity problems in ${\mathbb R},$  we will use the fact that 
the constant $\gamma$ in (\ref{viscoelastic_matrix}) 
is independent of $\varepsilon.$ 

The next step of our approach is to study the asymptotic behaviour as $\epsilon\to 0$ of the inverse of the operator 
\begin{equation}\label{e:mo}
    M\bigg(\frac{\cdot}{\eps}\bigg) - \begin{pmatrix} 0 & \partial_x \\[0.25em] \partial_x & 0 \end{pmatrix}
\end{equation}
in $\bigl[L^2(\mathbb{R})\bigr]^2$ for $M\in \mathcal{M}_\gamma$. Our strategy is based on applying the Gelfand transform (see \cite{Gelfand}) to the expression (\ref{e:mo}) and analysing the asymptotic behaviour of the elements of the associated fibre decomposition as $\varepsilon\to0.$

\begin{definition} Let $Q\coloneqq [0,1)$, $Q'\coloneqq [-\pi,\pi)$. For all $\epsilon>0$ we define the Gelfand transform (see \cite{Gelfand})
\[
   {\mathcal G}_\varepsilon \colon L^2(\mathbb{R})\to L^2(\epsilon^{-1}Q'\times Q)
\]
as the continuous extension to $L^2({\mathbb R})$ of the mapping given by
\[
   ({\mathcal G}_\varepsilon f)(\theta,y)\coloneqq \frac{\eps}{\sqrt{2\pi}}\sum_{n\in\mathbb{Z}} f\bigl(\epsilon(y+n)\bigr)e^{-{\rm i}\epsilon\theta(y+n)},\quad\ f\in C_{\rm c}^\infty(\mathbb{R}),\ \theta\in \epsilon^{-1}Q',\ y\in Q.
\] 
\end{definition}

Using the fact that the mapping ${\mathcal G}_\varepsilon$ is unitary, we rewrite the operator in \eqref{e:mo} according to the following result, see \cite[Lemma 3.1]{Cherednichenko2015}.
\begin{lemma}
\label{l:mog} 
Let $\gamma>0$, $M\in \mathcal{M}_\gamma$. Then for $\epsilon>0$ one has
\[
    {\mathcal G}_\varepsilon\Bigg(M\bigg(\frac{\cdot}{\eps}\bigg) - \begin{pmatrix} 0 & \partial_x \\[0.25em] \partial_x & 0 \end{pmatrix}\Bigg)^{-1}{\mathcal G}_\varepsilon^* = \int_{\epsilon^{-1}Q'}^\oplus \Bigg(M(\cdot) - \frac{1}{\epsilon}\begin{pmatrix} 0 & \partial_{\#}+{\rm i}\epsilon\theta \\[0.25em] \partial_{\#}+{\rm i}\epsilon\theta & 0 \end{pmatrix}\Bigg)^{-1} {\rm d}\theta,
\] 
where the operators under the integral on the right-hand side are defined and bounded on the space $\bigl[L^2(Q)\bigr]^2$ for each $\theta\in\varepsilon^{-1}Q'.$
\end{lemma}

In view of the above lemma, we are now concerned with the asymptotic behaviour, as $\varepsilon\to0,$ of the resolvents
\begin{equation}
 \Bigg(M - \frac{1}{\epsilon}\begin{pmatrix} 0 & \partial_{\#}+{\rm i}\epsilon\theta \\[0.25em] \partial_{\#}+{\rm i}\epsilon\theta & 0 \end{pmatrix}\Bigg)^{-1},\ \ \ \ \ \ \ \theta\in\varepsilon^{-1}Q',
 \label{fibre_inverse}
\end{equation}
where for convenience we drop the reference to the spatial argument of $M.$ 
In Section \ref{proof_thm} we establish the following result, which is similar in spirit to \cite[Theorem 7.1]{Cherednichenko2015}.

\begin{theorem}\label{t:mt} Suppose that  $\gamma>0$, $M\in \mathcal{M}_\gamma$. Then there exist $\eps'>0,$ $K>0$ such that 
\[
\Bigl\| \big( M-\epsilon^{-1}A_{\eps\theta}\big)^{-1} -\big(\av(M)-\epsilon^{-1}A_{\eps\theta}\big)^{-1}\Bigr\|_{{\mathbb L}([L^2(0,1)]^2)}\leq K\epsilon\ \ \ \ \ \ \ \ (\eps\le\eps',\ \ \theta\in \epsilon^{-1}Q'),
\]
where 
\begin{equation}
A_{\tau}\coloneqq \begin{pmatrix} 0 & \partial_{\#}+{\rm i}\tau \\[0.25em] \partial_{\#}+{\rm i}\tau & 0 \end{pmatrix},\ \ \ \ \ \tau\in Q'. 
\label{star0}
\end{equation}

Further, for the case when 
\begin{equation}
M=({\rm i}\xi+\nu)\begin{pmatrix}
      1 & 0 \\[0.3em] 0 & \bigl(\widehat{C}+\widehat{D}({\rm i}\xi+\nu)^\alpha\bigr)^{-1}
   \end{pmatrix},\ \ \ \ \ \ \ \xi \in \R,
\label{visco_matrix0}
\end{equation}
$\nu>\mu$ and $\widehat{C},$ $\widehat{D}$ satisfy the conditions discussed above (where $\mu\geq \nu_0$), the estimate
\begin{equation}
K\le \kappa(\xi^2+1),\ \ \ \xi\in{\mathbb R},
\label{xi_estimate}
\end{equation}
holds for some $\kappa>0$ independent of $\xi\in{\mathbb R}$ and $\varepsilon'$.
\end{theorem}

Note that the above theorem is in line with the result of \cite{Waurick2014SIAM_HomFrac}, where the homogenised coefficient is also given by $\av(M)$. The proof of Theorem \ref{t:mt} requires some preliminary work, which is essentially concerned with understanding the structure of the operator (\ref{fibre_inverse}) with respect to the subspace of $\bigl[L^2(Q)\bigr]^2$ consisting of vectors whose components belong to the null-space of the operator $\partial_{\#},$ {\it i.e.} constant functions on $Q.$ 
We start our analysis by discussing basic properties of the operator $A_{\eps\theta}$ defined by (\ref{star0})
and deriving its representation as a direct sum with respect to the associated orthogonal decomposition of the space $\bigl[L^2(Q)\bigr]^2.$ 

Combining Theorem \ref{t:mt} and Lemma \ref{l:mog}, we obtain the following norm-resolvent estimate in $[L^2({\mathbb R})]^2.$
\begin{corollary}
Let $\gamma>0$, $M\in \mathcal{M}_\gamma.$ Then there exist $\eps'>0,$ $\widetilde{K}=\widetilde{K}(\gamma)>0$ such that
\[
\Bigg\| \Bigg(M(\cdot/\epsilon) - \begin{pmatrix} 0 & \partial_{x} \\[0.25em] \partial_{x} & 0 \end{pmatrix}\Bigg)^{-1}- \Bigg(\av(M) - \begin{pmatrix} 0 & \partial_{x} \\[0.25em] \partial_{x} & 0 \end{pmatrix}\Bigg)^{-1}\Bigg\|_{{\mathbb L}([L^2(\mathbb{R})]^2)}\leq\widetilde{K}\epsilon\ \ \ \ \ \ \ (\eps\le \eps').
\]
\end{corollary}

\section{Orthogonal decomposition for the operator $A_{\epsilon\theta},$ $\theta\in\varepsilon^{-1}Q$}
\label{sec:ortho_decom}

In this section, we fix $\epsilon>0$ and $\theta\in \epsilon^{-1}Q',$ so that $\eps\theta=:\tau\in Q'.$ As it has been observed in \cite{Cherednichenko2015}, the nullspace of the operator $A_{\tau}$
is of primary importance to compute the limit problem. For purposes of the present work, it suffices to describe
the behaviour of $A_{\varepsilon \theta}$ in terms of the spectral subspaces of $\begin{pmatrix} 0 & \partial_{\#} \\ \partial_{\#}& 0 \end{pmatrix}$. 
To this end, we gather some well-known facts on the one-dimensional derivative with periodic boundary conditions, including the Poincar\'e--Wirtinger inequality, see \emph{e.g.} \cite{Almansi1905}.

\begin{proposition}
\label{l:1d} 
Consider the operator
\[
  \partial_\# \colon L^2(0,1)\supseteq H^1_\#(0,1)\ni f\mapsto f'\in L^2(0,1),
\]
where $H_\#^1(0,1)=\bigl\{f\in H^1(0,1): f(0)=f(1)\bigr\}$. Then the following assertions hold.
\begin{enumerate}
 \item The null-space of the operator $\partial_\#$ consists of constants: $N(\partial_\#)=
 \mathbb{C}.
 $
 \item There exists $c>0$ such that 
 \[
    \biggl\|f-\int_0^1f
    \biggr\|_{L^2(0,1)}\leq c\|f'\|_{L^2(0,1)}\ \ \ \ \bigl(f\in H_\#^1(0,1)\bigr).
 \]
 \item The orthogonal projection $Pf\in N(\partial_\#)^\bot$ of $f\in L^2(0,1)$ is given by
 \[
    Pf = f- \int_0^1 f.
 \]
 \item The range  $R(\partial_\#)$ of the operator $\partial_\#$ is closed in $L^2(0,1).$
 \item The operator $\partial_\#$ has compact resolvent and $\sigma(\partial_\#)=\{2\pi k{\rm i}; k\in \mathbb{Z}\}$.
 \item The operator $\partial_{\#,{\rm r}} \colon R(\partial_\#)\supseteq H^1_\#(0,1) \cap R(\partial_\#)\ni f\mapsto f'\in R(\partial_\#)$
  is continuously invertible.
\end{enumerate} 
\end{proposition}
\begin{proof}
Claim (a) is a straightforward observation and claim (b) is the standard Poincar\'e--Wirtinger inequality. To establish claim (c), it suffices to observe that for all $\alpha\in N(\partial_\#)$ we have
 \[
    \biggl\langle \alpha,f-\int_0^1f\biggr\rangle = \alpha^* \int_0^1\biggl(f(t)-\int_0^1f\biggr)dt = 0,
 \]
 which characterises the projection $Pf$. For claim (d), let $f_n
\in H_\#^1(0,1),$ $n\in{\mathbb N}$ be such that $f'_n\to
g\in L^2(0,1)$. Then, by the second property above, the functions  
$f_n-\int_0^1 f_n$ converge in $L^2(0,1)$ 
to some $h\in L^2(0,1)$. Observe that $h_n'=f_n'$ for all $n\in \mathbb{N}$. Since $\partial_\#$ is a closed operator, it follows that $h\in D(\partial_\#)$ and $h'=g\in R(\partial_\#)$.
 The operator $\partial_\#$ has compact resolvent by the Arzela--Ascoli Theorem (see {\it e.g.} \cite[Theorem I.28]{Reed_SimonI}). Hence, the spectrum of $\partial_\#$ solely consists of a countable set of eigenvalues, and since $\partial_\#=-\partial_\#^*$, we infer that it is a subset of the imaginary axis. Thus, ${\rm i}\lambda\in \sigma(\partial_\#)$ if and only if $\lambda\in \R$ and $\phi'={\rm i}\lambda \phi$ for some $\phi\in H_\#^1(0,1)$. Clearly, the last eqaution is solvable if and only if $\lambda\in 2\pi \mathbb{Z},$ which yields claim (e). Finally, claim (f) is a consequence of the fact that $R(\partial_\#)$ is closed.
 \end{proof}

 As a corollary of Proposition \ref{l:1d}, we obtain an alternative representation of the operators $\partial_\#+{\rm i}\eps\theta,$
 $\theta\in\varepsilon^{-1}Q.$ We use the reasoning employed in \cite{Waurick2014MMAS_Frac}.
 \begin{definition} The canonical injection from the nullspace of $N(\partial_\#)$ into $L^2(0,1)$ is denoted by
 \[
    \iota_{\textnormal{n}} \colon N(\partial_\#) \hookrightarrow L^2(0,1),\ \ f\mapsto f.
 \]
 Further, we set
 \[
    \iota_{\textnormal{r}} \colon R(\partial_\#)  \hookrightarrow L^2(0,1),\ \ f\mapsto f.
 \]  
 \end{definition}
  \begin{remark} It is straightforward to see ({\it cf.} \cite[Lemma 3.2]{Waurick2014MMAS_Frac}) that 
  \[
     \iota_{\textnormal{r}}^*\colon L^2(0,1)\to R(\partial_\#)
  \]
  is the orthogonal projection \emph{onto} $R(\partial_\#)$. The orthogonal projection with target space $L^2(0,1)$ acts as $\iota_{\textnormal{r}}^*$ and is given by $\iota_{\textnormal{r}}\iota_{\textnormal{r}}^*$.
 Likewise, $\iota_{\textnormal{n}}^*$ is the projection onto $N(\partial_\#),$ and by Lemma \ref{l:1d} the surjective operator $\iota_{\textnormal{n}}^*$ act as
  \[
     \iota_{\textnormal{n}}^*f= \int_0^1f, \quad\quad f\in L^2(0,1).
  \]
  From $L^2(0,1)=N(\partial_\#)\oplus R(\partial_\#)$, it follows that
  \[
     1 = \iota_{\textnormal{n}}\iota_{\textnormal{n}}^*+  \iota_{\textnormal{r}}\iota_{\textnormal{r}}^*.
  \]
  \end{remark}

 \begin{corollary}\label{c:rep} We have
 \[
     \begin{pmatrix}\iota_{\textnormal{n}}^* \\[0.3em] \iota_{\textnormal{r}}^*\end{pmatrix}\big(\partial_\# + {\rm i}\tau\big) \begin{pmatrix}\iota_{\textnormal{n}} & \iota_{\textnormal{r}}\end{pmatrix}  = \begin{pmatrix}
                                         {\rm i}\tau & 0 \\[0.3em] 0 & \partial_{\#,{\rm r}}+{\rm i}\tau
                                      \end{pmatrix}, \quad\quad\tau\in Q',
                                      \]
                                      as an operator in the Hilbert space $N(\partial_\#)\oplus R(\partial_\#)=L^2(0,1).$
 \end{corollary}
 \begin{proof}
    Note that the operator of multiplication by the constant ${\rm i}\tau$ leaves both $N(\partial_\#)$ and $R(\partial_\#)$ invariant. Hence,
    \begin{equation}
     \begin{pmatrix}\iota_{\textnormal{n}}^* \\[0.3em] \iota_{\textnormal{r}}^*\end{pmatrix}{\rm i}\epsilon\theta \begin{pmatrix}\iota_{\textnormal{n}} & \iota_{\textnormal{r}}\end{pmatrix}  = \begin{pmatrix}
                                         {\rm i}\epsilon\theta & 0 \\[0.25em] 0 & {\rm i}\epsilon\theta
                                      \end{pmatrix}.
                                      \label{star1}
    \end{equation}
    Further, on $N(\partial_\#)$ we have $\partial_\# =0,$ hence ({\it cf.} the last property in Lemma \ref{l:1d}) we obtain
    \begin{equation}
     \begin{pmatrix}\iota_{\textnormal{n}}^* \\[0.3em] \iota_{\textnormal{n}}^*\end{pmatrix}\partial_\#\begin{pmatrix}\iota_{\textnormal{n}} & \iota_{\textnormal{r}}\end{pmatrix}  = \begin{pmatrix}
                                         0 & 0 \\[0.25em] 0 & \partial_{\#,{\rm r}},
                                      \end{pmatrix},   
       \label{star2}
     \end{equation}
 where $\partial_{\#,{\rm r}}$ is the operator defined in Proposition \ref{l:1d}(f). Combining (\ref{star1}) and (\ref{star2}) yields the assertion. \end{proof}
\begin{remark}\label{r:prt}
  We note here that by Proposition \ref{l:1d}, we obtain that for $\tau\in [-\pi,\pi)$ one has $0\in \rho(\partial_{\#,{\rm r}}+{\rm i}\tau)$. Moreover, we get
  \[
     \bigl\|(\partial_{\#,{\rm r}}+{\rm i}\tau)^{-1}\bigr\|_{{\mathbb L}(R_\#)}
     \leq\bigl(\dist\bigl(\sigma(\partial_{\#,{\rm r}}),-{\rm i}\tau\bigr)\bigr)^{-1}
     \leq \frac{1}{\pi}.
  \]
\end{remark}
                                      
The result corresponding to Corollary \ref{c:rep} for the $(2\times 2)$-block operator matrix that corresponds to $A_{\tau},$ $\tau\in Q',$ reads as follows:
\begin{theorem}\label{t:nsp} The nullspace of the operator
\begin{align*}
   \begin{pmatrix} 0 & \partial_{\#} \\[0.25em] \partial_{\#} & 0 \end{pmatrix}\colon [L^2(0,1)]^2\supseteq[H_\#^1(0,1)]^2\ni (f,g)\mapsto (g',f')\in &  [L^2(0,1)]^2
\end{align*}
is given by
\[
   N_\#\coloneqq {\mathbb C}^2
   \subseteq [L^2(0,1)]^2.
\] 
We denote
\[
   {R}_\#\coloneqq[R(\partial_\#)]^2\subseteq [L^2(0,1)]^2,\ \ \ 
\]
and
\[
  \iota_N\coloneqq \begin{pmatrix}
                     \iota_{\textnormal{n}} & 0 \\[0.25em] 0 & \iota_{\textnormal{n}}
                   \end{pmatrix},\quad \iota_R\coloneqq \begin{pmatrix}
                     \iota_{\textnormal{r}} & 0 \\[0.25em] 0 & \iota_{\textnormal{r}}
                   \end{pmatrix},
\]
which form the canonical injections from the spaces $N_\#={R}_\#^\bot$ and ${R}_\#$ into $[L^2(0,1)]^2,$ respectively. 

Then one has 
\[
   \begin{pmatrix}\iota_{N}^* \\[0.3em] \iota_{R}^*\end{pmatrix}A_{\tau}\begin{pmatrix}\iota_{N} & \iota_{R}\end{pmatrix} = \begin{pmatrix} \begin{pmatrix} 0 & {\rm i}\tau\\[0.25em] {\rm i}\tau & 0 \end{pmatrix} & \begin{pmatrix} 0 & 0\\[0.25em] 0& 0 \end{pmatrix} \\[1.2em] \begin{pmatrix} 0 & 0 \\[0.25em] 0 & 0 \end{pmatrix} & \begin{pmatrix} 0 & \partial_{\#,{\rm r}}+{\rm i}\tau \\[0.25em] \partial_{\#,{\rm r}}+{\rm i}\tau & 0 \end{pmatrix} \end{pmatrix},\quad\quad\tau\in Q',
\]
where the right-hand side is treated as an operator in the space 
\[
\big(N(\partial_\#)\oplus N(\partial_\#)\big) \oplus \big(R(\partial_\#)\oplus R(\partial_\#)\big)={N}_\#\oplus{R}_\#=[L^2(0,1)]^2.
\]
\end{theorem}
\begin{proof}The proof is a straightforward consequence of Corollary \ref{c:rep} and Lemma \ref{l:1d}.
\end{proof}

\section{Averaging operator ${\tt av}$ and the null-space of $A_{\epsilon\theta}$ for $\epsilon=0,$ $\theta=0$}
\label{averaging}

In this section we establish the relationship between the averaging operator $\av$ 
and the projections introduced in the previous section, see in particular Theorem \ref{t:nsp}. 

\begin{theorem}\label{t:av} Suppose that $\gamma>0$, $M\in \mathcal{M}_\gamma$. Then the operator equality
\[
   \begin{pmatrix} \iota_{\textnormal{n}} & 0 \\[0.3em] 0 & \iota_{\textnormal{n}}
                   \end{pmatrix}^*M\begin{pmatrix} \iota_{\textnormal{n}} & 0 \\[0.3em] 0 & \iota_{\textnormal{n}}
                   \end{pmatrix} = \av(M)\begin{pmatrix} \iota_{\textnormal{n}} & 0 \\[0.3em] 0 & \iota_{\textnormal{n}}
                   \end{pmatrix}^*\begin{pmatrix} \iota_{\textnormal{n}} & 0 \\[0.3em] 0 & \iota_{\textnormal{n}}
                   \end{pmatrix}
\] 
holds.
\end{theorem}
\begin{proof}
Denote by $\pi_\textnormal{n}\coloneqq \begin{pmatrix} \iota_{\textnormal{n}} & 0 \\ 0 & \iota_{\textnormal{n}}
                   \end{pmatrix}\begin{pmatrix} \iota_{\textnormal{n}} & 0 \\ 0 & \iota_{\textnormal{n}}
                   \end{pmatrix}^*$ the orthogonal projection on $N_\#$ as an operator from $\bigl[L^2(0,1)\bigr]^2$ into itself.
 The assertion of the theorem is equivalent to the equality
 \[
    \pi_{\textnormal{n}} M\pi_{\textnormal{n}}=\av(M).
 \]
Hence, by the characterisation of the nullspace $N_\#$ in Theorem \ref{t:nsp}, we need to check whether for all $\alpha,\beta,\gamma,\delta\in \mathbb{C}$ one has
 \begin{equation}
    \Biggl\langle M \begin{pmatrix} \alpha\\[0.2em] \beta\end{pmatrix},\begin{pmatrix} \gamma\\[0.2em] \delta\end{pmatrix} \Biggr\rangle_{[L^2(0,1)]^2} = \Biggl\langle \av(M) \begin{pmatrix} \alpha\\[0.2em] \beta\end{pmatrix},\begin{pmatrix} \gamma\\[0.2em] \delta\end{pmatrix}\Biggr\rangle_{[L^2(0,1)]^2}
    \label{identity_check}
 \end{equation}
 holds. The identity (\ref{identity_check}) is verified directly and the assertion follows.
\end{proof}

For the analysis to follow, we record some simple facts.

\begin{lemma}\label{l:avinv} Suppose that $\gamma>0$, $M\in \mathcal{M}_\gamma$. Then $\av(M)\in \mathcal{M}_\gamma$. 
\end{lemma}
\begin{proof}
 Note that $\Re\av(M)=\av(\Re M)$ and that $\av(M)\geq\av(N)$ provided $M(x)\geq N(x),$ in the sense that $M(x)-N(x)$ is non-negative, for a.e. $x\in \mathbb{R}$. The latter two observations together with $c1_{2\times 2}\in \mathcal{M}_\gamma$ imply the assertion.
\end{proof}

\begin{lemma}\label{l:avdia} Suppose that $\gamma>0$, $M\in \mathcal{M}_\gamma$. Then one has
\[
   \iota_R^*\av(M) \iota_R = \av(M) \iota_R^*\iota_R.
\] 
\end{lemma}
\begin{proof} The assertion is a straightforward consequence of the definition of $\av(M)$.
\end{proof}

\begin{proposition}\label{t:avf} Suppose that $\gamma>0$, $M\in \mathcal{M}_\gamma$ with $\av(M)=M$. Then one has
\[
   \begin{pmatrix}\iota_{N}^* \\[0.4em] \iota_{R}^*\end{pmatrix} M \begin{pmatrix}\iota_{N} & \iota_{R}\end{pmatrix} = \begin{pmatrix}M & 0 \\[0.4em] 0 & M\end{pmatrix}.
\] 
\end{proposition}
\begin{proof}
 The assertion is a consequence of Theorem \ref{t:av} and Lemma \ref{l:avdia}.
\end{proof}

\section{Proof of Theorem \ref{t:mt}}
\label{proof_thm}

In this section we prove Theorem \ref{t:mt} and, therefore, we assume throughout the hypotheses of Theorem \ref{t:mt}, namely $\gamma>0$, $M\in \mathcal{M}_\gamma$. 

We require the following elementary result on the inverse of an operator in terms of the inverse of the Schur complements of an invertible operator on a subspace.

\begin{lemma}\label{l:s} Let $H_1,H_2$ be Hilbert spaces, $X\in {\mathbb L}(H_1)$, $Y\colon H_1\supseteq D(Y)\to H_2$, $\gamma_{12}\in {\mathbb L}(H_2,H_1)$, $\gamma_{21}\in {\mathbb L}(H_1,H_2)$. Assume that $0\in \rho\bigl(X-\gamma_{12}Y^{-1}\gamma_{21}\bigr)\cap \rho(Y)$. Then the operator 
\[
   \begin{pmatrix}
     X & \gamma_{12} \\[0.4em] \gamma_{21} & Y
   \end{pmatrix}
\]
is continuously invertible in ${\mathbb L}(H_1\oplus H_2)$ and for its inverse the following formula holds: 
\[
  \begin{pmatrix}
     X & \gamma_{12} \\[0.4em] \gamma_{21} & Y
   \end{pmatrix}^{-1} = \begin{pmatrix}
     I & 0 \\[0.4em] -Y^{-1}\gamma_{21} & I
   \end{pmatrix} \begin{pmatrix}
    \bigl (X-\gamma_{12}Y^{-1}\gamma_{21}\bigr)^{-1} & 0 \\[0.4em] 0 & Y^{-1}
   \end{pmatrix} \begin{pmatrix}
     I & -\gamma_{12}Y^{-1} \\[0.4em] 0 & I
   \end{pmatrix},
\]
where $I$ in the diagonal entries denotes the identity operators in the appropriate spaces.
\end{lemma}
\begin{proof}
 The proof is obtained by direct computation.
\end{proof}

The above lemma together with Theorem \ref{t:nsp} yields a representation for the operator $(\epsilon M - A_{\epsilon\theta})^{-1}$. For brevity, we introduce the operators ({\it cf.} (\ref{star0}))
\[
 A_{\tau,{\rm r}}\coloneqq \begin{pmatrix} 0 & \partial_{\#,{\rm r}}+{\rm i}\tau \\[0.4em]   \partial_{\#,{\rm r}}+{\rm i}\tau & 0\end{pmatrix},\quad\quad
 B_{\tau}\coloneqq \begin{pmatrix} 0 & {\rm i}\tau\\[0.4em] {\rm i}\tau& 0\end{pmatrix},\quad\quad Y_\tau\coloneqq \epsilon \iota_R^*M\iota_R- A_{\tau, {\rm r}},\quad\quad\tau\in Q',
\]
\[
\Gamma_{11}\coloneqq \iota_N^* M \iota_N,\ \ \ \Gamma_{12}\coloneqq \iota_N^*M\iota_R,\ \ \ \Gamma_{21}\coloneqq \iota_R^*M\iota_N,\ \ \ \Gamma_{22}\coloneqq \iota_R^*M\iota_R,\ \ \ 
\]
and also use the following notation:
\[
\epsilon_0:=\frac{\pi}{2\|M\|},\quad\quad\epsilon_1:=\frac{\gamma\pi}{4\|M\|^2},\quad\quad\epsilon':=\min\{\epsilon_0,\epsilon_1\}.
\]

\begin{theorem}\label{t:rinv} 
For all $\epsilon\in (0,\epsilon')$ one has
\begin{multline*}
   \begin{pmatrix} \iota_N^*\\[0.4em] \iota_R^*\end{pmatrix}(\epsilon M - A_{\epsilon\theta})^{-1}\begin{pmatrix} \iota_N & \iota_R\end{pmatrix} \\=
   \begin{pmatrix} I & 0 \\[0.4em] - \epsilon Y_{\varepsilon\theta}^{-1}\Gamma_{21} & I\end{pmatrix}
   \begin{pmatrix}\bigl(\epsilon \Gamma_{11} - B_{\epsilon\theta} -\epsilon^2\Gamma_{12}Y_{\varepsilon\theta}^{-1}\Gamma_{21}\bigr)^{-1} & 0 \\[0.4em] 0 & Y_{\varepsilon\theta}^{-1}\end{pmatrix}
   \begin{pmatrix} I & - \epsilon\Gamma_{12}Y_{\varepsilon\theta}^{-1} \\[0.4em] 0 & I\end{pmatrix},\quad\quad\theta\in\varepsilon^{-1}Q'.
\end{multline*}
Furthermore, the following bound holds:
\begin{equation}
  \sup_{{\epsilon\in (0,\epsilon'),}\atop{\theta\in \epsilon^{-1}[-\pi,\pi)}}\bigl\|Y_{\varepsilon\theta}^{-1}\bigr\|_{{\mathbb L}(R_\#)}\leq \frac{2}{\pi}.
  \label{bound_Y_inverse}
\end{equation}
\end{theorem}

The proof of Theorem \ref{t:rinv} relies on Lemma \ref{l:s},  where we set $\gamma_{21}=\varepsilon \Gamma_{21},$ $\gamma_{12}=\varepsilon\Gamma_{12},$ $X=\varepsilon\Gamma_{11}-B_{\varepsilon\theta},$ $Y=Y_{\varepsilon\theta}.$ 
First, we address the invertibility of the operator  $Y_{\varepsilon\theta}=\epsilon \iota_R^*M\iota_R- A_{\epsilon\theta,{\rm r}}.$

\begin{theorem}\label{t:Y} 
The operator $Y$ is continuously invertible and the bound (\ref{bound_Y_inverse}) holds.
\end{theorem}
\begin{proof}
 By Remark \ref{r:prt}, the operator $A_{\tau, {\rm r}}$
 is continuously invertible for every $\tau\in Q'$. More precisely, we obtain
 \[
  A_{\tau, {\rm r}}^{-1}= (\partial_{\#,{\rm r}}+{\rm i}\tau)^{-1}\begin{pmatrix} 0 & 1 \\[0.3em] 1 & 0 \end{pmatrix},\quad\quad\quad
  \bigl\|A_{\tau, {\rm r}}^{-1}
    \bigr\|_{{\mathbb L}(R_\#)}\leq \frac{1}{\pi},\quad\quad\quad\tau\in Q'.
 \]
Furthermore, for $\epsilon\in(0,\epsilon_0)$, $\theta\in \epsilon^{-1}[-\pi,\pi),$ by a Neumann series argument (see {\it e.g.} \cite[pp. 30--34]{Kato}), the operator
\[
Y_{\varepsilon\theta}=A_{\epsilon\theta,{\rm r}}\bigl(\epsilon A_{\epsilon\theta,{\rm r}}^{-1}\iota_R^*M \iota_R-1\bigr)
\]
 is a composition of continuously invertible operators, and 
 \begin{align*}
   \bigl\|Y_{\varepsilon\theta}^{-1}\bigr\|_{{\mathbb L}(R_\#)}
   & 
   =\bigl\|(\epsilon A_{\epsilon\theta,{\rm r}}^{-1}\iota_R^*M \iota_R-1)^{-1}A_{\epsilon\theta,{\rm r}}^{-1}\bigr\|_{{\mathbb L}(R_\#)} 
   \\&    
   \leq \frac{1}{\pi}\sum_{k=0}^\infty\bigl\|\epsilon A_{\epsilon\theta,{\rm r}}^{-1}\iota_R^*M \iota_R\bigr\|_{{\mathbb L}(R_\#)}^k
   \\
  & \leq \frac{1}{\pi}\sum_{k=0}^\infty \bigg(\epsilon_0 \frac{\|M\|}{\pi}\bigg)^k
    \leq \frac{1}{\pi}\sum_{k=0}^\infty \bigg(\frac{1}{2}\bigg)^k=\frac{2}{\pi} ,
 \end{align*}
which yields the claim.
\end{proof}
 
 Next, we discuss the invertibility of the term $\epsilon \Gamma_{11}- B_{\epsilon\theta} -\epsilon^2\Gamma_{12}Y_{\varepsilon\theta}^{-1}\Gamma_{21}$ in Theorem \ref{t:rinv}.
\begin{lemma}\label{l:X} 
For all $\epsilon\in(0, \epsilon'),$ $\theta\in\varepsilon^{-1}Q',$ the operator
\[
X_{\varepsilon\theta}\coloneqq  \Gamma_{11}-\epsilon^{-1}B_{\epsilon\theta} -\epsilon\Gamma_{12}Y_{\varepsilon\theta}^{-1}\Gamma_{21}
\]  
is continuously invertible, and the following bound holds:
\[
   \sup_{{\epsilon\in(0, \epsilon'),}
   \atop{\theta\in \epsilon^{-1}[-\pi,\pi)}}\bigl\|X_{\varepsilon\theta}^{-1}\bigr\|_{{\mathbb L}(N_\#)}\leq \frac{2}{\gamma}.
\] 
\end{lemma}
\begin{proof}
 By Theorem \ref{t:Y}, the operator norm of  $Y_{\varepsilon\theta}^{-1}$ is bounded uniformly in $\epsilon\in(0,\epsilon_0)$
  and $\theta\in \epsilon^{-1}[-\pi,\pi)$. Furthermore, notice that $\Re(\Gamma_{11}-\epsilon^{-1}B_{\epsilon\theta})=\Re \Gamma_{11}\geq \gamma$. Hence, 
  \begin{equation}
  \bigl\|(\Gamma_{11}-\epsilon^{-1}B_{\epsilon\theta})^{-1}\bigr\|_{{\mathbb L}(N_\#)}\leq \frac{1}{\gamma},
  \label{newstar}
  \end{equation}
   which yields the applicability of a Neumann series argument for the invertibility of $X_{\varepsilon\theta}.$ Taking into account the fact that $\|Y_{\varepsilon\theta}^{-1}\|\leq 2/\pi$ and arguing in a similar way to the proof of Theorem \ref{t:Y}, we obtain
  the invertibility of $X_{\varepsilon\theta}$ as well as the bound $\bigl\|X_{\varepsilon\theta}^{-1}\bigr\|\leq 2/\gamma$ for all $\epsilon\in (0, \epsilon').$ 
\end{proof}

\begin{proof}[Proof of Theorem \ref{t:rinv}] The assertion follows by combining Theorem \ref{t:Y}, Lemma \ref{l:X} and Theorem \ref{t:nsp}. 
\end{proof}

\begin{proof}[Proof of Theorem \ref{t:mt}] We apply Theorem \ref{t:rinv} to $M\in \mathcal{M}_\gamma$
and consider the difference 
\[
   D_\tau\coloneqq \begin{pmatrix} \iota_N^*\\[0.4em] \iota_R^*\end{pmatrix}\bigl(\epsilon M - A_\tau\bigr)^{-1}\begin{pmatrix} \iota_N & \iota_R\end{pmatrix} - \begin{pmatrix} \iota_N^*\\[0.4em] \iota_R^*\end{pmatrix}\bigl(\epsilon \av(M) - A_\tau\bigr)^{-1}\begin{pmatrix} \iota_N & \iota_R\end{pmatrix},\quad\quad\tau\in Q'.
\]
First, we show that 
\begin{equation}\label{eq:aim}
\sup_{{\epsilon\in (0,\epsilon'),}\atop{\theta\in \epsilon^{-1}[-\pi,\pi)}} \|D_{\epsilon\theta}\|_{{\mathbb L}([L^2(0,1)]^2)}\leq \max\biggl\{\frac{4\gamma^2}{\pi}\|M\|^2,\ \frac{2}{\pi}\left(\frac{2\|M\|}{\gamma}+1\right)\biggr\} \eqqcolon K(\|M\|).
\end{equation}

Taking Proposition \ref{t:avf} into account, we infer from Theorem \ref{t:rinv} that
 \[
    \begin{pmatrix} \iota_N^*\\[0.4em] \iota_R^*\end{pmatrix}\bigl(\epsilon \av(M) - A_{\epsilon\theta}\bigr)^{-1}\begin{pmatrix} \iota_N & \iota_R\end{pmatrix}
    =   \begin{pmatrix}\bigl(\epsilon \Gamma_{11}-B_{\epsilon\theta}\bigr)^{-1} & 0 \\[0.4em] 0 & {\widetilde Y}_{\varepsilon\theta}^{-1}\end{pmatrix},
 \]
where $\widetilde Y: = \epsilon \iota_R^*\av(M)\iota_R- A_{\epsilon\theta,{\rm r}}$. Hence, Theorem \ref{t:rinv} yields
\begin{align*}
   D_{\epsilon\theta}
    & = 
   \begin{pmatrix} 1 & 0 \\[0.4em] - \epsilon Y_{\varepsilon\theta}^{-1}\Gamma_{21} & 1\end{pmatrix}
   \begin{pmatrix}\bigl(\epsilon \Gamma_{11} - B_{\epsilon\theta} -\epsilon^2\Gamma_{12}Y_{\varepsilon\theta}^{-1}\Gamma_{21}\bigr)^{-1} & 0 \\[0.4em] 0 & Y_{\varepsilon\theta}^{-1}\end{pmatrix}
   \begin{pmatrix} 1 & - \epsilon\Gamma_{12}Y_{\varepsilon\theta}^{-1} \\[0.4em] 0 & 1\end{pmatrix}
    \\ &
    \quad  \quad \quad  \quad    \quad  \quad
    - \begin{pmatrix} (\epsilon \Gamma_{11}-B_{\epsilon\theta})^{-1} & 0 \\[0.4em] 0 & {\widetilde Y}_{\varepsilon\theta}^{-1}\end{pmatrix}
    \\ & = \begin{pmatrix} 1 & 0 \\[0.4em] - \epsilon Y_{\varepsilon\theta}^{-1}\Gamma_{21} & 1\end{pmatrix}
    \Bigg\{\begin{pmatrix}\bigl(\epsilon \Gamma_{11} - B_{\epsilon\theta} -\epsilon^2\Gamma_{12}Y_{\varepsilon\theta}^{-1}\Gamma_{21}\bigr)^{-1} & 0 \\[0.4em] 0 & Y_{\varepsilon\theta}^{-1}\end{pmatrix} \\ & \quad\quad\quad\quad\quad\quad - \begin{pmatrix} 1 & 0 \\[0.4em]  \epsilon Y_{\varepsilon\theta}^{-1}\Gamma_{21} & 1\end{pmatrix}\begin{pmatrix} (\epsilon \Gamma_{11}-B_{\epsilon\theta})^{-1} & 0 \\[0.4em] 0 & {\widetilde Y}_{\varepsilon\theta}^{-1}\end{pmatrix}\begin{pmatrix} 1 &  \epsilon\Gamma_{12}Y_{\varepsilon\theta}^{-1} \\[0.4em] 0 & 1\end{pmatrix} \Bigg\}
    \begin{pmatrix} 1 & - \epsilon\Gamma_{12}Y_{\varepsilon\theta}^{-1} \\[0.4em] 0 & 1\end{pmatrix}.
\end{align*}
 By Theorem \ref{t:Y}, for all $\epsilon\in (0,\epsilon'),\theta\in \epsilon^{-1}Q'$ we have $\Big\|\Big(\begin{smallmatrix} 1 &  -\epsilon\Gamma_{12}Y_{\varepsilon\theta}^{-1} \\ 0 & 1\end{smallmatrix}\Big)\Big\|\leq 1+2\epsilon\|M\|/\pi\leq 2$ and, hence,
 \begin{multline*}
     \|D_{\epsilon\theta}\|_{{\mathbb L}([L^2(0,1)]^2)}\leq 4\Bigg\| \begin{pmatrix}\bigl(\epsilon \Gamma_{11} - B_{\epsilon\theta} -\epsilon^2\Gamma_{12}Y_{\varepsilon\theta}^{-1}\Gamma_{21}\bigr)^{-1} & 0 \\[0.4em] 0 & Y_{\varepsilon\theta}^{-1}\end{pmatrix} \\  - \begin{pmatrix} 1 & 0 \\[0.4em]  \epsilon Y_{\varepsilon\theta}^{-1}\Gamma_{21} & 1\end{pmatrix}\begin{pmatrix}\bigl(\epsilon \Gamma_{11}-B_{\epsilon\theta}\bigr)^{-1} & 0 \\[0.4em] 0 & {\widetilde Y}_{\varepsilon\theta}^{-1}\end{pmatrix}\begin{pmatrix} 1 &  \epsilon\Gamma_{12}Y_{\varepsilon\theta}^{-1} \\[0.4em] 0 & 1\end{pmatrix} \Bigg\|_{{\mathbb L}([L^2(0,1)]^2)}.
 \end{multline*}
 Further, noting that
 \begin{align*}
   &\begin{pmatrix} 1 & 0 \\[0.4em]  \epsilon Y_{\varepsilon\theta}^{-1}\Gamma_{21} & 1\end{pmatrix}\begin{pmatrix} (\epsilon \Gamma_{11}-B_{\epsilon\theta})^{-1} & 0 \\[0.4em] 0 & {\widetilde Y}_{\varepsilon\theta}^{-1}\end{pmatrix}\begin{pmatrix} 1 &  \epsilon\Gamma_{12}Y_{\varepsilon\theta}^{-1} \\[0.4em] 0 & 1\end{pmatrix} \\[0.5em]
   &=\begin{pmatrix} (\epsilon \Gamma_{11}-B_{\epsilon\theta})^{-1} & 0 \\[0.4em] \epsilon Y_{\varepsilon\theta}^{-1}\Gamma_{21}(\epsilon \Gamma_{11}-B_{\epsilon\theta})^{-1} & {\widetilde Y}_{\varepsilon\theta}^{-1}\end{pmatrix}\begin{pmatrix} 1 &  \epsilon\Gamma_{12}Y_{\varepsilon\theta}^{-1} \\[0.4em] 0 & 1\end{pmatrix} \\[0.5em]
   &=\begin{pmatrix} (\epsilon \Gamma_{11}-B_{\epsilon\theta})^{-1} & (\epsilon \Gamma_{11}-B_{\epsilon\theta})^{-1}\epsilon\Gamma_{12}Y_{\varepsilon\theta}^{-1} \\[0.4em] \epsilon Y_{\varepsilon\theta}^{-1}\Gamma_{21}(\epsilon \Gamma_{11}-B_{\epsilon\theta})^{-1} & {\widetilde Y}_{\varepsilon\theta}^{-1}+\epsilon Y_{\varepsilon\theta}^{-1}\Gamma_{21}(\epsilon \Gamma_{11}-B_{\epsilon\theta})^{-1}\epsilon\Gamma_{12}Y_{\varepsilon\theta}^{-1} \end{pmatrix} \\[0.5em]
   &= \begin{pmatrix} (\epsilon \Gamma_{11}-B_{\epsilon\theta})^{-1} &\bigl( \Gamma_{11}-\epsilon^{-1}B_{\epsilon\theta}\bigr)^{-1}\Gamma_{12}Y_{\varepsilon\theta}^{-1} \\[0.4em]  Y_{\varepsilon\theta}^{-1}\Gamma_{21}\bigl(\Gamma_{11}-\epsilon^{-1}B_{\epsilon\theta}\bigr)^{-1} & {\widetilde Y}_{\varepsilon\theta}^{-1}+\epsilon Y_{\varepsilon\theta}^{-1}\Gamma_{21}\bigl(\Gamma_{11}-\epsilon^{-1}B_{\epsilon\theta}\bigr)^{-1}\Gamma_{12}Y_{\varepsilon\theta}^{-1}\end{pmatrix},
 \end{align*}
we estimate
 \begin{multline*}
     \sup_{{\epsilon\in (0,\epsilon_1),}\atop{\theta\in \epsilon^{-1}[-\pi,\pi)}} \Bigg\| \begin{pmatrix} 0 & \bigl(\Gamma_{11}-\epsilon^{-1}B_{\epsilon\theta}\bigr)^{-1}\Gamma_{12}Y_{\varepsilon\theta}^{-1} \\[0.5em]  Y_{\varepsilon\theta}^{-1}\Gamma_{21}\bigl(\Gamma_{11}-\epsilon^{-1}B_{\epsilon\theta}\bigr)^{-1} & {\widetilde Y}_{\varepsilon\theta}^{-1}+\epsilon Y_{\varepsilon\theta}^{-1}\Gamma_{21}\bigl(\Gamma_{11}-\epsilon^{-1}B_{\epsilon\theta}\bigr)^{-1}\Gamma_{12}Y_{\varepsilon\theta}^{-1}\end{pmatrix}\Bigg\|_{{\mathbb L}([L^2(0,1)]^2)}
     \\ \leq\frac{2\|M\|}{\gamma\pi}+\frac{2}{\pi}+\frac{\epsilon}{\gamma}\left(\frac{2\|M\|}{\pi}\right)^2
      \leq \frac{2}{\pi}\left(\frac{2\|M\|}{\gamma}+1\right).
 \end{multline*}
 Indeed, the latter follows from the bound (\ref{newstar})
 as well as Theorem \ref{t:Y} (applied to both $Y_{\varepsilon\theta}$ and ${\widetilde Y}_{\varepsilon\theta}:$ recall Lemma \ref{l:avinv}, to deduce that $\av(M)\in \mathcal{M}_\gamma,$ and Theorem \ref{t:avf}). Hence, in order to obtain \eqref{eq:aim}, it remains to show that 
 \begin{equation}\label{e:faim}
    \sup_{{\epsilon\in (0,\epsilon'),}\atop{\theta\in \epsilon^{-1}[-\pi, \pi)}}\Bigl\|\bigl(\epsilon \Gamma_{11} - B_{\epsilon\theta} -\epsilon^2\Gamma_{12}Y_{\varepsilon\theta}^{-1}\Gamma_{21}\bigr)^{-1}-(\epsilon \Gamma_{11}-B_{\epsilon\theta})^{-1}\Bigr\|_{{\mathbb L}(N_\#)}\leq K\bigl(\|M\|\bigr).
 \end{equation} In view of the fact that
 \begin{align*}
    & (\epsilon \Gamma_{11} - B_{\epsilon\theta} -\epsilon^2\Gamma_{12}Y_{\varepsilon\theta}^{-1}\Gamma_{21})^{-1} 
 \\[0.35em]
 & =\Bigl(1- \bigl(\Gamma_{11} - \epsilon^{-1}B_{\epsilon\theta}\bigr)^{-1}\epsilon\Gamma_{12}Y_{\varepsilon\theta}^{-1}\Gamma_{21}\Bigr)^{-1}(\epsilon \Gamma_{11} - B_{\epsilon\theta})^{-1}
    \\ 
    & = \sum_{k=0}^\infty \Big((\Gamma_{11} - \epsilon^{-1}B_{\epsilon\theta})^{-1}\Gamma_{12}Y_{\varepsilon\theta}^{-1}\Gamma_{21}\epsilon\Big)^k(\epsilon \Gamma_{11} - B_{\epsilon\theta})^{-1} 
    \\ & = (\epsilon \Gamma_{11} - B_{\epsilon\theta})^{-1}+ \sum_{k=1}^\infty \Big(\bigl(\Gamma_{11} -\epsilon^{-1}B_{\epsilon\theta})^{-1}\Gamma_{12}Y_{\varepsilon\theta}^{-1}\Gamma_{21}\Big)^k\epsilon^{k-1}\bigl(\Gamma_{11} - \epsilon^{-1}B_{\epsilon\theta}\bigr)^{-1},
 \end{align*}
 the estimate \eqref{e:faim} follows. Indeed, one has
 \[
 \left\|\sum_{k=1}^\infty \Big(\bigl(\Gamma_{11} -\epsilon^{-1}B_{\epsilon\theta})^{-1}\Gamma_{12}Y_{\varepsilon\theta}^{-1}\Gamma_{21}\Big)^k\epsilon^{k-1}\bigl(\Gamma_{11} - \epsilon^{-1}B_{\epsilon\theta}\bigr)^{-1}\right\|_{{\mathbb L}(N_\#)}
 \quad\quad\quad\quad\quad\quad\quad\quad\quad\quad
\]
\[ 
\quad\quad\quad\quad\quad\quad\quad\quad\quad\quad
\leq\frac{1}{\gamma}\sum_{k=1}^\infty \Big(\frac{2}{\pi}\|M\|^2\Big)^k\epsilon^{k-1} \leq\frac{4\gamma^2}{\pi}\|M\|^2, 
\]
and, thus, also the estimate \eqref{eq:aim} follows. On the other hand, note that by the skew-selfadjointness of $A_\tau,$ 
$\tau\in Q',$ we infer that
\[
    \Bigl\| \big( M-\epsilon^{-1}A_\tau\big)^{-1} -\big(\av(M)-\epsilon^{-1}A_\tau\big)^{-1}\Bigr\|_{{\mathbb L}([L^2(0,1)]^2)}\leq \frac{2}{\gamma},
\]
for all $\tau\in Q.$
Thus, there exist $\kappa_1,\kappa_2\geq0$ independent of $\|M\|$ and $\epsilon$ such that
\[
     \Bigl\| \big( M-\epsilon^{-1}A_{\eps\theta}\big)^{-1} -\big(\av(M)-\epsilon^{-1}A_{\eps\theta}\big)^{-1}\Bigr\|_{{\mathbb L}([L^2(0,1)]^2)}\leq \begin{cases} \kappa_1(\|M\|^2+1),& 0<\epsilon\leq(\kappa_2 \|M\|^2)^{-1},\\[0.3em]
       2/\gamma,& \epsilon>0.\end{cases}
\]
Hence, we obtain the existence of some $\kappa\geq0$ independent of $\|M\|$ and $\epsilon$ such that 
\[
   \Bigl\| \big( M-\epsilon^{-1}A_{\eps\theta}\big)^{-1} -\big(\av(M)-\epsilon^{-1}A_{\eps\theta}\big)^{-1}\Bigr\|_{{\mathbb L}([L^2(0,1)]^2)}\leq \kappa\bigl(\|M\|^2+1\bigr)\epsilon.\qedhere
\]
\end{proof}

\section{Resolvent convergence of solutions to the viscoelasticity problem}
\label{viscoelasticity_treatment}

\subsection{Operator-norm resolvent estimates in ${\mathbb L}\bigl(H^1_\nu\bigl(\mathbb{R}; L^2({\mathbb R})\bigr), L_\nu^2\bigl(\mathbb{R}; L^2({\mathbb R})\bigr)\bigr)$}

Here we consider the problem of fractional elasticity ({\it cf.} (\ref{n_dependent}))
\begin{equation}
      \partial_t^2 u_\varepsilon- \partial_x \bigl(C(\cdot/\varepsilon) + \partial_t^\alpha D(\cdot/\varepsilon)\bigr)\partial_{x} u_\varepsilon=f,\ \ \ \ \ \varepsilon>0,\ \ \ f\in{L_\nu^2\bigl(\mathbb{R};L^2({\mathbb R})\bigr)},
\label{viscoelastic_eps}
\end{equation}
under the same assumptions on the coefficients $C,D$ and exponent $\alpha$ as in Corollary \ref{c:st}.
Using the Fourier--Laplace transform  ${\mathcal L}_\nu$ introduced in Section \ref{well_pos_section} and the fibre decomposition of Lemma \ref{l:mog} we first write an expression for the vector consisting of the solution to (\ref{viscoelastic_eps}) and its flux (whose role is played by the stress in the viscoelastic medium)
that allows us to apply Theorem \ref{t:mt} directly.  
Namely, for a given right-hand side $f$ in (\ref{viscoelastic_eps}),
denoting by 
\[
v_\varepsilon:=\partial_tu_\varepsilon, \quad\quad\quad\sigma_\varepsilon:=\bigl(C(\cdot/\varepsilon) + \partial_t^\alpha D(\cdot/\varepsilon)\bigr)\partial_{x} u_\varepsilon
\]
the velocity and viscoelastic stress at each point of the medium, we obtain, for all $\nu>\mu,$
\begin{equation}
\left(\begin{array}{c}v_\varepsilon \\ \sigma_\varepsilon\end{array}\right)=\Biggl(\mathcal{L}_\nu^*\Biggl(\overline{({\rm i}m+\nu)\begin{pmatrix}
      1 & 0 \\[0.3em] 0 & \bigl(C(\cdot/\varepsilon)+D(\cdot/\varepsilon)({\rm i}m+\nu)^\alpha\bigr)^{-1}
   \end{pmatrix} - \begin{pmatrix} 0 & \partial_x \\[0.3em]  \partial_x & 0
   \end{pmatrix}}\Biggr)\mathcal{L}_\nu\Biggr)^{-1}\left(\begin{array}{c}f \\ 0\end{array}\right)
   \label{bar1}
\end{equation}
\begin{equation}
=\mathcal{L}_\nu^*\,
\overline{({\rm i}m+\nu)\begin{pmatrix}
      1 & 0 \\[0.3em] 0 & \bigl(C(\cdot/\varepsilon)+D(\cdot/\varepsilon)({\rm i}m+\nu)^\alpha\bigr)^{-1}
   \end{pmatrix} - \begin{pmatrix} 0 & \partial_x \\[0.3em]  \partial_x & 0
   \end{pmatrix}}
   ^{-1}\mathcal{L}_\nu\left(\begin{array}{c}f \\ 0\end{array}\right)
\label{bar2}
\end{equation}
\[
\quad=\mathcal{L}_\nu^*{\mathcal G}_\varepsilon^*
 \int_{\epsilon^{-1}Q'}^\oplus\biggl[\overline{
 \Big(
 M_m^{\rm ve}(\cdot) - \varepsilon^{-1}A_{\varepsilon\theta}
 \Big)}
 ^{-1}-\overline{
 \Big(
 {\tt av}\bigl(M_m^{\rm ve}\bigr)- \varepsilon^{-1}A_{\varepsilon\theta}
 \Big)}
 ^{-1}\biggr]{\rm d}\theta\,{\mathcal G}_\varepsilon\mathcal{L}_\nu\left(\begin{array}{c}f \\ 0\end{array}\right)
\]
\[
\quad\quad\quad\quad\quad\quad\quad\quad\quad\quad\quad\quad+\mathcal{L}_\nu^*{\mathcal G}_\varepsilon^*
\int_{\epsilon^{-1}Q'}^\oplus \overline{
\Big(
{\tt av}\bigl(M_m^{\rm ve}\bigr) - \varepsilon^{-1}A_{\varepsilon\theta}
\Big)}
^{-1} {\rm d}\theta\,{\mathcal G}_\varepsilon\mathcal{L}_\nu\left(\begin{array}{c}f \\ 0\end{array}\right),
\]
where  
\[
M_\xi^{\rm ve}(\cdot):=({\rm i}\xi+\nu)\begin{pmatrix}
      1 & 0 \\[0.3em] 0 & \bigl(\widehat{C}(\cdot)+\widehat{D}(\cdot)({\rm i}\xi+\nu)^\alpha\bigr)^{-1}
   \end{pmatrix},\ \ \ \ \xi\in{\mathbb R},
\]
is the expression for the matrix $M$ in the general theory of Sections \ref{Gelfand_section}, \ref{averaging}, \ref{proof_thm}, $\mu\ge\nu_0$ such that (\ref{viscoelastic_matrix}) holds with $M(\cdot)=M_\xi^{\rm ve}(\cdot),$ {\it cf.} 
(\ref{visco_matrix0}), and the bar in (\ref{bar1}), (\ref{bar2}), as before, denotes the closure of the operator. Furthermore, it follows from Theorem \ref{t:mt} with $M=M_\xi^{\rm ve},$ that there exists a constant $\tilde \kappa>0$ such that
\begin{multline*}
\Biggl\|\partial_t^{-1}{\mathcal I}{\mathcal P}\mathcal{L}_\nu^*{\mathcal G}_\varepsilon^*
 \int_{\epsilon^{-1}Q'}^\oplus\biggl[\overline{
 \Big(
 M_m^{\rm ve}(\cdot) - \varepsilon^{-1}A_{\varepsilon\theta}
 \Big)}
 ^{-1}\Biggr.
 \\[0.5em]
 \Biggl.-\overline{
 \Big(
 {\tt av}\bigl(M_m^{\rm ve}\bigr)- \varepsilon^{-1}A_{\varepsilon\theta}
 \Big)}
 ^{-1}\biggr]{\rm d}\theta\,{\mathcal G}_\varepsilon\mathcal{L}_\nu{\mathcal P}^*{\mathcal I}^*\Biggr\|_{{\mathbb L}(H^1_\nu(\mathbb{R}; L^2({\mathbb R})), L_\nu^2(\mathbb{R}; L^2({\mathbb R})))}\leq\tilde\kappa\epsilon,
 \end{multline*}
where ${\mathcal P}$ is the projection on the subspace of vectors with vanishing second component, and ${\mathcal I}$ is the isomorphism 
\[
{\mathcal I}:\ L_\nu^2\bigl(\mathbb{R}; [L^2({\mathbb R})]^2\bigr)\ni\left(\begin{array}{c}f\\0\end{array}\right)\mapsto f\in L_\nu^2\bigl(\mathbb{R}; L^2({\mathbb R})\bigr).
\]
 We have thus proved the following theorem.
 \begin{theorem} 
 \label{res_est}
 Under the hypotheses in this section, there exists $\tilde\kappa>0$  such that for all $\varepsilon>0$ one has
  \begin{multline*}
   \left\|\partial_t^{-1}{\mathcal I}{\mathcal P}\left[
   \overline{
   \partial_t\begin{pmatrix}
      1 & 0 \\[0.3em] 0 & \bigl(\widehat C+\widehat D\partial_t^\alpha\bigr)^{-1}
   \end{pmatrix} + \begin{pmatrix} 0& \partial_x \\[0.3em] \partial_x & 0
   \end{pmatrix}
   }^{-1}\right.\right.\\[0.5em] 
  \left.\left.
\phantom{\overline{\left(\begin{pmatrix}
       1^2 \\0^2\end{pmatrix}\right)}}  
 -\mathcal{L}_\nu^*{\mathcal G}_\varepsilon^*
 \int_{\epsilon^{-1}Q'}^\oplus
 \overline{
  \Big(
  {\tt av}\bigl(M^{\rm ve}\bigr)- \varepsilon^{-1}A_{\varepsilon\theta}
  \Big)}
  ^{-1}{\rm d}\theta\,{\mathcal G}_\varepsilon\mathcal{L}_\nu\right]{\mathcal P}^*{\mathcal I}^*\right\|_{{\mathbb L}(H^1_\nu(\mathbb{R}; L^2({\mathbb R})), L_\nu^2(\mathbb{R}; L^2({\mathbb R})))}\leq\tilde\kappa\epsilon.
\end{multline*}
 \end{theorem}

\subsection{Estimates in Littlewood--Paley type spaces}

\begin{definition}
 Suppose that $\alpha\geq 0$, $\nu>0$, and let $H$ be a Hilbert space. Then for $f\in L_\nu^2(\mathbb{R};H)$ we denote
  \[
    \| f \|_{\alpha,\nu}^2 \coloneqq \sum_{k\in \mathbb{Z}} e^{-|k|\alpha}\bigl\|\mathcal{L}_\nu^*(m_k)\mathcal{L}_\nu f\bigr\|_{\nu}^2,
  \]
  where $\mathcal{L}_\nu$ is the Fourier--Laplace transform, see (\ref{Laplace_transform}), and $m_k$ is the operator of multiplication by the characteristic function of the interval $[k,k+1)$. We also define
  \[
     LP_\nu(\alpha)\coloneqq\bigl( L_\nu^2(\mathbb{R};H), \|\cdot\|_{\alpha,\nu}\bigr)^\sim,
  \]
which we refer to as the \emph{Littlewood--Paley space} with growth $\alpha$.
\end{definition}
\begin{remark} We list some basic properties of the Littlewood--Paley spaces:
   
   1) $LP_\nu(\alpha)$ is a Hilbert space for all $\alpha\geq 0$, $\nu>0$.
   
   2) $LP_\nu(0)= L_\nu^2(\mathbb{R};H)$ for all $\nu>0.$ 
   
   3) For all $\alpha, \beta\in [0,\infty)$ one has $LP_\nu(\alpha)\hookrightarrow LP_\nu(\beta)$ whenever $\alpha\leq \beta.$
   
   4) For all $\alpha\geq 0$, $\nu>0$ the embedding $LP_\nu(0)\hookrightarrow LP_\nu(\alpha)$ has dense range. 
\end{remark}

\begin{definition}
  Let $\nu>0$, $T\in {\mathbb L}(L_\nu^2(\mathbb{R};H))$. We say that $T$ is \emph{translation-invariant}, if for all $h\in \mathbb{R}$, one has  $\tau_hT=T\tau_h$, where $\tau_h f\coloneqq f(\cdot + h)$. A translation-invariant operator $T$ is called \emph{(forward) causal}, if for all $f\in L_\nu^2(\mathbb{R};H)$ with $f=0$ on $(-\infty,0]$ one has $Tf=0$ on $(-\infty,0]$.
\end{definition}

In order to refine the estimate of Theorem \ref{res_est} in operator norms associated with Littlewood--Paley spaces, we need the following general property of translation invariant and causal maps in weighted spaces. 

\begin{theorem}[{{{\it cf.~e.g.}~\cite[Corollary 1.2.5]{Waurick2016Hab}} or \cite{Weiss1991}}]\label{th:SW_weighted} Let $H$ Hilbert space, $\nu>0$, $T\in {\mathbb L}(L_{\nu}^2(\mathbb{R};H))$ translation-invariant and causal. Then $T$ extends to an operator in ${\mathbb L}(L_{\rho}^2(\mathbb{R};H))$ for all $\rho>\nu$ and there is a unique operator-valued function $\mathcal{T}\colon \mathbb{C}_{\Re>\nu}\to {\mathbb L}(H)$ satisfying
\begin{equation}\label{eq:SW_w1}
   (\mathcal{L}_{\rho}Tu)(\xi)=\mathcal{T}\left(i\xi+\rho\right)\mathcal{L}_{\rho}u (\xi),\quad\quad(\xi\in \mathbb{R},\ \rho>\nu,\ \ u\in L_\rho^2(\mathbb{R};H)).
\end{equation}
 Moreover, the function $\mathcal{T}$ is bounded, analytic, and  
 \[
    \|{T}\|_{{\mathbb L}(L_{\rho}^2(\mathbb{R};H))}\le\sup_{z\in \mathbb{C}_{\Re\geq \nu}}\|\mathcal{T}(z)\|_{{\mathbb L}(H)}.
 \] 
\end{theorem}

\begin{proposition} 
\label{extension}
Suppose that $\nu>0$, $\beta\geq\alpha\geq 0$, and let $T\in {\mathbb L}(L_\nu^2(\mathbb{R};H))$ be translation-invariant and causal. Then for all $\rho>\nu$ the operator $T$ admits a unique continuous extension to a mapping $T_{\alpha,\beta}\in {\mathbb L}(LP_\rho(\alpha),LP_\rho(\beta))$, such that 
\[
\|T_{\alpha,\beta}\|_{{\mathbb L}(LP_\rho(\alpha), LP_\rho(\beta))}\leq\|T\|_{{\mathbb L}(L_\nu^2(\mathbb{R};H))}.
\] 
\end{proposition}
\begin{proof}  Let $\mathcal{T}$ be the analytic operator-valued function representing $T$ from Theorem \ref{th:SW_weighted}. We denote by $\mathcal{T}_{m,\rho}$ the operator on $L^2(\mathbb{R};H)$ of multiplication by the mapping $\xi \mapsto \mathcal{T}(i\xi+\rho)$. Then, for all $f\in LP_\nu(\alpha)$ one has 
\begin{align*}
   \|Tf\|_{\beta,\nu}^2& = \sum_{k\in \mathbb{Z}} e^{-|k|\beta}\|\mathcal{L}_\nu^* m_k \mathcal{L}_\nu Tf \|_\nu^2
   \\ & = \sum_{k\in\mathbb{Z}} e^{-|k|\beta}\|\mathcal{L}_\nu^* m_k \mathcal{L}_\nu \mathcal{L}_\nu^* \mathcal{T}_{m,\rho}\mathcal{L}_\nu f \|_\nu^2
   \\ & = \sum_{k\in\mathbb{Z}} e^{-|k|\beta}\|\mathcal{L}_\nu^* m_k \mathcal{T}_{m,\rho}\mathcal{L}_\nu f \|_\nu^2
   \\ & = \sum_{k\in\mathbb{Z}} e^{-|k|\beta}\|\mathcal{L}_\nu^* m_k \mathcal{T}_{m,\rho}m_k\mathcal{L}_\nu f \|_\nu^2
   \\ & = \sum_{k\in\mathbb{Z}} e^{-|k|\beta}\|\mathcal{L}_\nu^* m_k \mathcal{T}_{m,\rho}\mathcal{L}_\nu \mathcal{L}_\nu^* m_k\mathcal{L}_\nu f \|_\nu^2
   \\ & \leq \sum_{k\in\mathbb{Z}} e^{-|k|\beta}\sup_{\xi\in[k,k+1)}\| m_k \mathcal{T}(i\xi+\nu)\|_{{\mathbb L}(H)}\|\mathcal{L}_\nu^* m_k\mathcal{L}_\nu f \|_\nu^2
   \\ & \leq \sum_{k\in\mathbb{Z}} e^{-|k|\beta}\|{T}\|^2_{{\mathbb L}(L_{\nu}^2(\mathbb{R};H))}\|\mathcal{L}_\nu^* m_k\mathcal{L}_\nu f \|_\nu^2
   \\ & = \|{T}\|^2_{{\mathbb L}(L_{\nu}^2(\mathbb{R};H))}\sum_{k\in\mathbb{Z}} e^{-|k|\alpha}\|\mathcal{L}_\nu^* m_k\mathcal{L}_\nu f \|_\nu^2
   \\ & = \|{T}\|^2_{{\mathbb L}(L_{\nu}^2(\mathbb{R};H))}\|f\|_{\beta,\nu}^2\leq \|{T}\|^2_{{\mathbb L}(L_{\nu}^2(\mathbb{R};H))}\|f\|_{\alpha,\nu}^2.\qedhere
\end{align*}
\end{proof}

The next theorem asserts that one can get a quantified estimate for the difference of two translation-invariant and causal operators as operators in the Littlewood--Paley spaces. This has -- as it will be demonstrated below -- applications in the theory of quantitative homogenisation theory, where it is possible to obtain quantitative (resolvent) estimates that are uniform only on compact subsets of the resolvent parameter. From the applied perspective, one may therefore 
think of $T$ and $S$ in the following theorem as the solution operators to certain partial differential equations in space-time. 

\begin{theorem} 
\label{difference}
Suppose that $\nu>0,$ $\beta>\alpha\ge0$ and operators $T,S\in {\mathbb L}(L_\nu^2(\mathbb{R};H))$ are translation-invariant and causal. consider the operator-valued function $\mathcal{T}$ and $\mathcal{S}$ representing $T$ and $S$ respectively, as in Theorem \ref{th:SW_weighted}. Assume there exist $\kappa>0$ and $\eta\geq 0$ such that 
\[
\|\mathcal{T}(i \xi + \rho)-\mathcal{S}(i \xi + \rho)\|_{{\mathbb L}(H)}\leq \kappa (|\xi|+1)^\eta\quad\quad(\xi\in \mathbb{R}).
\]
Then for all $\rho>\nu$ the estimate 
\[
\|T_{\alpha,\beta} - S_{\alpha,\beta}\|_{{\mathbb L}(LP_\rho(\alpha), LP_\rho(\beta))} \leq C\kappa,
\] 
holds, where 
\begin{equation}
C=C(\alpha,\beta,\eta)\coloneqq \max_{k\in \mathbb{Z}_{\geq0}} e^{(\alpha-\beta)k}(k+2)^\eta<\infty.
\label{C_formula}
\end{equation}
\end{theorem}
\begin{proof}
 The claim of the theorem follows from the following estimate, valid  for all $f\in LP_\rho(\alpha):$ 
 \begin{align*}
   \bigl\|(T_{\alpha,\beta} - S_{\alpha,\beta})f\bigr\|_{\beta,\rho}^2& = \sum_{k\in\mathbb{Z}} e^{-\beta |k|}\bigl\|\mathcal{L}_\rho^* m_k\mathcal{L}_\rho(T_{\alpha,\beta} - S_{\alpha,\beta})f\bigr\|_\rho^2
   \\ & = \sum_{k\in\mathbb{Z}} e^{-\beta |k|}\bigl\|\mathcal{L}_\rho^* m_k (\mathcal{T}_{m,\rho}-\mathcal{S}_{m,\rho})\mathcal{L}_\rho f\bigr\|_\rho^2
   \\ & = \sum_{k\in\mathbb{Z}} e^{-\beta |k|}\bigl\|\mathcal{L}_\rho^* m_k (\mathcal{T}_{m,\rho}-\mathcal{S}_{m,\rho})\mathcal{L}_\rho\mathcal{L}_\rho^* m_k\mathcal{L}_\rho f\bigr\|_\rho^2 
   \\ & \leq \sum_{k\in\mathbb{Z}} e^{-\alpha |k|} e^{(\alpha-\beta)|k|} \sup_{\xi\in[k,k+1)}\bigl\|(\mathcal{T}(i \xi +\rho)-\mathcal{S}(i \xi +\rho))\|_{{\mathbb L}(H)}\|\mathcal{L}_\rho^* m_k\mathcal{L}_\rho f\bigr\|_\rho^2
   \\ & \leq \sum_{k\in\mathbb{Z}} e^{-\alpha |k|} e^{(\alpha-\beta)|k|} \sup_{\xi\in[k,k+1)} \kappa(|\xi|+1)^\eta\|\mathcal{L}_\rho^* m_k\mathcal{L}_\rho f\|_\rho^2
   \\ & \leq C\kappa\|f\|_{\alpha,\rho}^2.
 \end{align*}
\end{proof}

Applying Theorems \ref{extension} and \ref{difference} to the operators discussed in Theorem \ref{res_est} and bearing in mind the estimate (\ref{xi_estimate}) yields the following result.

\begin{corollary} For all $\rho>\nu,$ where $\nu$ is any value admissible in (\ref{bar1}), and $\beta>\alpha\ge0,$ the operators 
 \[
   {\mathcal R}:={\mathcal I}{\mathcal P}
   \overline{
\partial_t\begin{pmatrix}
      1 & 0 \\[0.3em] 0 & \bigl(\widehat C+\widehat D\partial_t^\alpha\bigr)^{-1}
   \end{pmatrix} + \begin{pmatrix} 0& \partial_x \\[0.3em] \partial_x & 0
   \end{pmatrix}
   }^{-1}{\mathcal P}^*{\mathcal I}^*,
   \]
   \[
 {\mathcal R}^{\rm hom}:={\mathcal I}{\mathcal P}\mathcal{L}_\nu^*{\mathcal G}_\varepsilon^*
 \int_{\epsilon^{-1}Q'}^\oplus
 \overline{
  \Big(
  {\tt av}\bigl(M^{\rm ve}\bigr)- \varepsilon^{-1}A_{\varepsilon\theta}
  \Big)}
  ^{-1}{\rm d}\theta\,{\mathcal G}_\varepsilon\mathcal{L}_\nu{\mathcal P}^*{\mathcal I}^*
  \]
have extensions as linear bounded operators from $LP_\rho(\alpha)$ to $LP_\rho(\beta)\bigr),$ and the estimate 
 \begin{equation}
   \bigl\|{\mathcal R}
 -{\mathcal R}^{\rm hom}
  \bigr\|_{{\mathbb L}(LP_\rho(\alpha), LP_\rho(\beta))}
  \leq C\kappa\epsilon
  \end{equation}
holds for all $\varepsilon>0,$
where $C=C(\alpha,\beta,2),$ see (\ref{C_formula}), and $\kappa$ is the constant in (\ref{xi_estimate}).
\end{corollary}

\section*{Acknowedgements} The authors are indebted to Shane Cooper for many valuable discussions on the manuscript.  The authors are grateful for the financial support of the Engineering and Physical Sciences Research Council (Grant EP/L018802/1 ``Mathematical foundations of metamaterials: homogenisation, dissipation and operator theory'').

\bibliographystyle{plain}

\noindent
Kirill Cherednichenko\\Department of Mathematical Sciences, University of Bath,\\
Claverton Down, Bath, BA2 7AY,\\ 
United Kingdom\\
Email: k.cherednichenko@bath.ac.uk

\medskip

\noindent
Marcus Waurick \\Department of Mathematical Sciences, University of Bath,\\
Claverton Down, Bath, BA2 7AY,\\
United Kingdom\\
Email: M.Waurick@bath.ac.uk

\end{document}